\documentclass[12pt]{amsart}

\usepackage{amsfonts}
\usepackage{amsmath,amscd,amssymb,amsthm}
\usepackage{mathrsfs}
\usepackage{pstricks}
\usepackage{pst-node}
\usepackage[left=1.1875in,top=1in,text={6.5in,9in}]{geometry}
\usepackage{url}
\usepackage{tikz}
\usepackage{tikz-cd}

\newtheorem{theorem}{Theorem}
\theoremstyle{plain}
\newtheorem{lemma}[theorem]{Lemma}
\newtheorem{proposition}[theorem]{Proposition}
\newtheorem{corollary}[theorem]{Corollary}

\theoremstyle{definition}

\theoremstyle{remark}

\newtheorem{remark}[theorem]{Remark}
\newtheorem{example}[theorem]{Example}

\newcommand{\Z}{\mathbb{Z}}
\newcommand{\F}{\mathbb{F}}

\newcommand{\ds}{\displaystyle}

\begin{document}

\title{Homology Operations in Symmetric Homology}
\author{Shaun V. Ault}
\email{svault@valdosta.edu}
\address{Department of Mathematics and Computer Science,
         Valdosta State University, 
         1500 N. Patterson St.
         Valdosta, Georgia, 31698, 
         USA.}



\begin{abstract}  
  The symmetric homology of a unital associative algebra $A$ over a
  commutative ground ring $k$, denoted $HS_*(A)$, is defined using
  derived functors and the symmetric bar construction of Fiedorowicz.
  In this paper we show that $HS_*(A)$ admits homology operations and
  a Pontryagin product structure making $HS_*(A)$ an associative
  commutative graded algebra.  This is done by finding an explicit
  $E_{\infty}$ structure on the standard chain groups that compute
  symmetric homology.
\end{abstract}

\maketitle

\section{Introduction}\label{sec.intro}                        

The purpose of this paper is to define an $E_{\infty}$ structure on
the standard chain groups that compute symmetric homology of a unital
associative algebra.  The construction makes use of the fact that the
symmetric category $\Delta S_+$ (that is, $\Delta S$ with an initial
object appended) is permutative, a property not shared by the
simplicial category $\Delta$ nor the cyclic category $\Delta C$, even
if initial objects are appended.  Such structure may facilitate
computations of symmetric homology, which in turn may shed light on
related functor homology theories.

The notion of symmetric homology was introduced under the broader
context of crossed simplicial groups (CSGs) by Fiedorowicz and Loday
in~\cite{FL}.  Some important properties and results were developed in
the preprints of Fiedorowicz~\cite{F} and Ault-Fiedorowicz~\cite{AF},
as well as in the author's thesis, a portion of which has been
published~\cite{A}.  Symmetric homology can be thought of as an analog
to cyclic homology, in which the symmetric groups play the role that
the cyclic groups do in the latter.  The usefulness of cyclic
(co)homology in noncommutative geometry and $K$-theory is well
established (see for example,~\cite{Connes,Burgh,BurghFied}). It
becomes natural to examine generalizations such as symmetric homology
in order to better understand cyclic homology itself.  Moreover, these
generalizations are important in their own right.  For example, there
are interesting links between symmetric homology and $\Gamma$-homology
and related theories through the identification of $\Delta S$ with the
category of {\it noncommutative sets}, $\mathcal{F}(as)$ (see
\S\ref{sub.noncomm_sets}).  Furthermore, symmetric homology is related
to stable homotopy theory in the following way: if $G$ is a group, the
symmetric homology of the group ring $k[G]$ is isomorphic to
$H_*(\Omega\Omega^{\infty} S^{\infty}(BG); k)$~\cite{F,A}.

This paper, together with~\cite{A}, is intended to supplant the
unpublished preprints of Fiedorowicz~\cite{F} and
Ault-Fiedorowicz~\cite{AF}.

\subsection{Acknowledgments.}

The author is grateful for the mentorship of his thesis advisor
Zbigniew Fiedorowicz, with whom he worked closely (as his
Ph.~D.~student at Ohio State) to develop key results of this paper.
The author is also indebted to the anonymous referees of previous
versions of this paper for many helpful suggestions, and to certain
faculty members at Fordham University for editorial guidance.

\subsection{Symmetric Homology}

We begin by recalling some of the notations and definitions regarding
symmetric homology found in~\cite{A}.  Let $A$ be a unital associative
algebra over a commutative ground ring $k$, and let $k$-$\mathbf{Mod}$
be the category of (left) $k$-modules.  Let $\Delta S$ be the category
whose objects are the sets $[n] = \{0, 1, 2, \ldots, n\}$ for $n \geq
0$ and whose morphisms $[n] \to [m]$ are pairs $(\phi, \gamma)$ such
that $\phi$ is a non-decreasing set map $[n] \to [m]$ (that is, $\phi
\in \Delta([n],[m])$), and $\gamma \in \Sigma^{\mathrm{op}}_{n+1}$
(the opposite of the symmetric group).  The category $\Delta S$ is the
structure category of the {\it symmetric CSG}~\cite{FL,L}.  Briefly, a
CSG is a sequence of groups $\{G_n\}_{n \geq 0}$ together with a
structure category $\Delta G$ such that:
\begin{itemize}
  \item $\Delta G$ contains the simplicial category $\Delta$ as
    subcategory,
  \item $\mathrm{Aut}_{\Delta G}([n]) = G_n^{\mathrm{op}}$, and
  \item Each morphism of $\Delta G$ has unique decomposition into
    $\phi \circ \gamma$, which we denote by the pair $(\phi, \gamma)$,
    with $\phi \in \Delta$ and $\gamma \in G_n$ for some $n$.
\end{itemize}
Composition in $\Delta G$ is defined by $(\phi, \gamma) \circ (\psi,
\delta) = (\phi \circ \psi^{\gamma}, \gamma^{\psi} \cdot \delta)$ for
the appropriate morphisms $\psi^{\gamma}$ of $\Delta$ and
${\gamma}^{\psi} \in G^{\mathrm{op}}$.  As implied by the notation, a
single dot ($\cdot$) is used for multiplication in $G^{\mathrm{op}}$,
however it is convenient to regard the group elements as living in $G$
so that we typically do the multiplication ``the right way'' when
writing the morphism: $(\phi, \gamma) \circ (\psi, \delta) =
(\phi\psi^{\gamma}, \delta \gamma^{\psi})$.  See~\cite{FL,A} for more
details and notational conventions.  Observe, both $\Delta$ and the
cyclic category $\Delta C$ are examples of structure categories of
CSGs, the former having trivial automorphism groups and the latter
having $\mathrm{Aut}_{\Delta C}([n]) = C_{n+1}^{\mathrm{op}} =
C_{n+1}$, the cyclic group of order $n+1$.  Using an appropriate bar
construction, one may define a homology theory associated to a CSG.
Indeed the cyclic bar construction of Loday~\cite{L}, a contravariant
functor $B_*^{cyc}A \colon\Delta C \to \textrm{$k$-$\mathbf{Mod}$}$,
defines the cyclic homology of the algebra $A$ via $HC_*(A) =
\mathrm{Tor}_*^{\Delta C^{\mathrm{op}}}(\underline{k}, B_*^{cyc}A) =
\mathrm{Tor}_*^{\Delta C}(B_*^{cyc}A, \underline{k})$, where
$\underline{k}$ is the trivial cyclic $k$-module, that is,
$\underline{k}[n] = k$ for all $n \geq 0$ and $\underline{k}\alpha =
\mathrm{id}$ for all morphisms $\alpha$.  Fiedorowicz and
Loday~\cite{FL} found that any definition of symmetric homology using
a contravariant bar construction results in a trivial theory -- that
is, if $M$ is a $\Delta S^{\mathrm{op}}$-module, then
$\mathrm{Tor}_*^{\Delta S}(M, \,\underline{k}) = H_*(M)$, the homology
of the underlyling simplicial module.  On the other hand,
Fiedorowicz~\cite{F} discovered that the {\it covariant} bar
construction, rather than a contravariant one, yields an interesting
non-trivial theory of symmetric homology.
\begin{equation}\label{eqn.symbar}
  \begin{array}{c}
  B_*^{sym}A \colon\Delta S \to \textrm{$k$-$\mathbf{Mod}$},\\
  B_*^{sym}A[n] = B_n^{sym}A \stackrel{def}{=} A^{\otimes (n + 1)}.
  \end{array}
\end{equation}
The functor $B_*^{sym}A$ is referred to as $C^{sym}(A)$ in~\cite{L}.
It is sufficient to define $B_*^{sym}A$ on $\gamma \in
\Sigma_{n+1}^{\mathrm{op}}$ and $\phi \in \Delta([n], [m])$.  We often
refer to the morphism $B_*^{sym}A\alpha$ as {\it evaluation at
  $\alpha$}.
\begin{eqnarray}
  \label{eqn.eval_g_1}
  B_*^{sym}A\gamma \colon A^{\otimes(n+1)} &\longrightarrow&
  A^{\otimes(n+1)} \\
  \label{eqn.eval_g_2}
  a_0 \otimes a_1 \otimes \cdots \otimes a_n &\mapsto& a_{\gamma(0)}
  \otimes a_{\gamma(1)} \otimes \cdots \otimes a_{\gamma(n)}.
\end{eqnarray}
\begin{eqnarray}
  \label{eqn.eval_phi_1}
  B_*^{sym}A\phi \colon A^{\otimes(n+1)} &\longrightarrow&
  A^{\otimes(m+1)} \\
  \label{eqn.eval_phi_2}
  a_0 \otimes a_1 \otimes \cdots \otimes a_n &\mapsto& b_0 \otimes b_1
  \otimes \cdots \otimes b_m, \quad \textrm{where} \\ b_i &=&
  \prod_{a_j \in \phi^{-1}(i)} a_j, \quad \textrm{(product taken in
    order of increasing indices)}
\end{eqnarray}

We define the {\it symmetric homology} of any $\Delta S$-module $M$ by
$HS_*(M) = \mathrm{Tor}_*^{\Delta S}( \underline{k}, M )$, in which
$\underline{k}$ is the trivial $\Delta S^{\mathrm{op}}$-module.  For
any unital associative algebra $A$, $B_*^{sym}A$ is a $\Delta
S$-module, and so we define the symmetric homology of $A$ as follows:
\begin{equation}
  HS_*(A) \stackrel{def}{=} \mathrm{Tor}_*^{\Delta
    S}\left(\underline{k}, B_*^{sym}A\right).
\end{equation}

It is advantageous to enlarge $\Delta S$ by adding an initial object
$[-1] \in \Delta S_+$.  Define the extended symmetric bar
construction, $B_*^{sym_+}A$, by $B_n^{sym_+}A = B_n^{sym}A$ for $n
\geq 0$ and $B_{-1}^{sym_+}A = k$.  Evaluation at the unique morphism
$[-1] \to [n]$ sends $1 \in k$ to $1^{\otimes(n+1)}$.  The author has
shown~\cite{A} that symmetric homology also can be computed using
$\Delta S_+$.  In~(\ref{eqn.extendedHS}), $\underline{k}$ is the
trivial $\Delta S_+^{\mathrm{op}}$-module.
\begin{equation}\label{eqn.extendedHS}
  HS_*(A) \cong \mathrm{Tor}_*^{\Delta S_+}\left( \underline{k},
  B_*^{sym_+}A \right).
\end{equation}
We may use a standard resolution based on under-categories to compute
the Tor groups.  Recall, for a small category $\mathscr{C}$, there is
a contravariant functor $- \setminus \mathscr{C}$ from $\mathscr{C}$
to $\mathbf{Cat}$ (the category of small categories), which takes an
object $c$ to the under-category $c \setminus \mathscr{C}$; in other
words, $- \setminus \mathscr{C}$ is a
$\mathscr{C}^{\mathrm{op}}$-category.  Using the notation
$N\mathscr{C}$ for the {\it nerve} of a small category $\mathscr{C}$,
and the useful notation of Gabriel-Zisman~\cite{GZ}, a simplicial
$k$-module whose homology is exactly $HS_*(A)$ is written and defined
as follows:
\begin{equation}\label{eqn.GZ-chain}
  C_*(\Delta S_+, B_*^{sym_+}A) \stackrel{def}{=} k\left[N(- \setminus
    \Delta S_+)\right] \otimes_{\Delta S_+} B_*^{sym_+}A.
\end{equation}
That is, 
\begin{equation}\label{eqn.GZ-HS(A)}
  HS_*(A) \cong H_*(C_*(\Delta S_+, B_*^{sym_+}A)).
\end{equation}

\section{Preliminaries}\label{sec.prelim}                        

\subsection{Notational Conventions}\label{sub.notation}
With an eye towards readability, we use the following notational
conventions:
\begin{enumerate}
  \item {\it Tuple of $n$ items:} $\mathbf{m} \stackrel{def}{=} (m_1,
    m_2, \ldots, m_n)$.  Each element $m_i$ may be a number or an
    element of some set as context dictates.  The number of elements,
    $n$, is suppressed in the notation, though it will always be clear
    what $n$ is by context.
  \item {\it ``Single-variable'' function applied to a tuple:} If $f
    \colon M \to N$ and $\mathbf{m} \in M^n$, then
    \begin{equation}
      f(\mathbf{m}) \stackrel{def}{=} (f(m_1), f(m_2), \ldots, f(m_n))
      \in N^n.
    \end{equation}
  \item {\it ``Multi-variable'' function applied to a tuple:} If $f
    \colon M^p \to N^q$, then we simply write the image of $\mathbf{m}
    \in M^p$ under $f$ as $f(\mathbf{m}) \in N^q$.
  \item {\it Permutation applied to a tuple:} If $\sigma \in \Sigma_n$,
    then 
    \begin{equation}
      \sigma\mathbf{m} \stackrel{def}{=} (m_{\sigma^{-1}(1)}, \ldots,
      m_{\sigma^{-1}(n)}).
    \end{equation}
    This convention ensures that $\Sigma_n$ acts on the {\it left} of
    $\mathbf{m}$.
  \item {\it Block permutation:} If $\sigma \in \Sigma_n$, then
    $\sigma_{\mathbf{k}} = \sigma_{k_1,\ldots, k_n} \in
    \Sigma_{k_1+\cdots + k_n}$ represents the {\it block
      transformation} of blocks of sizes $k_1, k_2, \ldots k_n$ (where
    each $k_i \in \mathbb{N} \cup \{0\}$).  For example, $(1,2)_{2,3}
    = (1,4,2,5,3)$.
  \item {\it Inter-block permutation:} If $\sigma_i \in \Sigma_{k_i}$
    for each $i=1, 2, \ldots, n$, then $\sigma_1 \oplus \cdots \oplus
    \sigma_n \in \Sigma_{k_1+\cdots+k_n}$ represents the permutation
    of block $i$ by $\sigma_i$ while retaining the original order of
    the blocks.  For example, $(1,2) \oplus (1,2,3) = (1,2)(3,4,5)$.
  \item {\it Products of tuples:} Suppose $c_i$ are objects of a
    category $\mathscr{C}$ with an associative binary operation
    $\odot$.  Then $\mathbf{c}^{\odot} \stackrel{def}{=} c_1 \odot c_2
    \odot \cdots \odot c_n$.  Moreover, if $\sigma \in \Sigma_n$, then
    $\sigma\mathbf{c}^{\odot} \stackrel{def}{=} c_{\sigma^{-1}(1)}
    \odot \cdots \odot c_{\sigma^{-1}(n)}$.
  \item If there is a specified left action of $\Sigma_n$ on a set
    $X$, then the notation $\sigma \bullet x$ denotes the image of $x
    \in X$ under the action of $\sigma \in \Sigma_n$.  The same
    notation is used for right actions, only written the opposite way
    around: $x \bullet \sigma$.  This notation is chosen so that there
    is a clear distinction between the similar notations $\sigma
    \mathbf{c}^{\odot}$ and $\sigma \bullet \mathbf{c}^{\odot}$.
\end{enumerate}

\subsection{Monoid Algebras and the Functor $\mathcal{T}$}\label{sub.monoids}
Let $\mathbf{Mon}$ be the category of monoids and monoid homomorphisms
(here, we mean {\it ordinary monoids in sets}).  For a given monoid
$M$, we define the extended symmetric bar construction,
\begin{equation}\label{eqn.symbar_monoid}
  \begin{array}{c}
  B^{sym_+}M \colon\Delta S_+ \to \mathbf{Mon},\\ B^{sym_+}M[n] 
  \stackrel{def}{=} M^{n+1}
  \end{array}
\end{equation}
Here, $M^n$ is the cartesian product of $n$ copies of $M$, and $M^0 =
\{ () \}$, a set containing just the empty tuple.  Now let us define a
similar notation as that of~(\ref{eqn.GZ-chain}) for simplicial
monoids.  If $F$ is a $\Delta S_+$-monoid, {\it i.e.}  $F : \Delta S_+
\to \mathbf{Mon}$ is a functor, then let
\begin{equation}\label{eqn.GZ-set}
  C(\Delta S_+, F) \stackrel{def}{=} N( - \setminus \Delta S_+)
  \times_{\Delta S_+} F.
\end{equation}
We define the symmetric homology (with coefficients in $k$) of the
monoid $M$ by:
\begin{equation}\label{eqn.GZ-simplicialset}
  HS_*(M) \stackrel{def}{=} H_*( k[C(\Delta S_+, B^{sym_+}M)] ).
\end{equation}
See~\cite{A}, \S 5.2, for more details on the symmetric bar
construction for monoids.

\begin{remark}
  We will generally use the notation $\langle f, \mathbf{m} \rangle$
  in place of $B^{sym_+}_*Mf(\mathbf{m})$ order to denote {\it
    evaluation} of $f$ at $\mathbf{m}$.  Because $B_*^{sym_+}M$ is
  functorial, the evaluation map satisfies the useful property,
  \begin{equation}\label{eqn.eval_prop}
    \langle fg, \mathbf{m} \rangle = \big\langle f, \langle g,
    \mathbf{m} \rangle \big\rangle.
  \end{equation}
  Similarly, for $a_0 \otimes \cdots \otimes a_n \in
  A^{\otimes(n+1)}$, we may write $\langle f, a_0 \otimes \cdots
  \otimes a_n \rangle$ in place of $B^{sym_+}_*Af(a_0 \otimes \cdots
  \otimes a_n)$.
\end{remark}

Let $\mathcal{T}$ be the functor from $\mathbf{Mon}$ to the category
of small categories defined by sending a monoid $M$ to the category
$\mathcal{T}M$ whose objects are finite sequences of elements of $M$,
including the empty sequence, $()$.  Morphisms of $\mathcal{T}M$
consist of pairs $(f, \mathbf{m})$ such that $\mathbf{m} = (m_1,
\ldots, m_p) \in M^p$ and $f \colon [p-1] \to [q-1]$ is a morphism of
$\Delta S_+$.  The source and target of such a pair are $\mathbf{m}$
and $\langle f, \mathbf{m}\rangle$, respectively. When the source and
target are clear, we simply use $f$ to denote the morphism.  The
functor $\mathcal{T}$ sends a monoid morphism $\psi \colon M \to N$ to
the functor $\mathcal{T}\psi \colon \mathcal{T}{M} \to \mathcal{T}{N}$
that maps $\mathbf{m} \in M^p$ to $\psi(\mathbf{m}) \in N^p$.

\begin{lemma}\label{lem.TM_permutative}
  $\mathcal{T}M$ is a permutative category.
\end{lemma}
\begin{proof}
  Define the product on objects by concatenation:
  \begin{equation}
    (m_1, \ldots, m_p) \odot (n_1, \ldots, n_q) \stackrel{def}{=}
    (m_1, \ldots, m_p, n_1, \ldots, n_q).
  \end{equation}
  Since $\Delta S_+$ is permutative~\cite{A}, we can use the product
  of $\Delta S_+$ to define products of morphisms in $\mathcal{T}M$.
  Associativity is strict, since it is induced by the associativity of
  $\odot$ in $\Delta S_+$.  The empty sequence, $()$, is a strict
  unit.  The symmetry transformation is defined on objects by block
  transposition.
\end{proof}

\subsection{The Category of Noncommutative Sets}\label{sub.noncomm_sets}

There is an interpretation of the morphisms of $\Delta S_+$ as formal
tensors, which provides an interesting connection to the category
$\mathcal{F}(as)$, the category of {\it noncommutative
  sets}~\cite{P,PR,Richter}.  The objects of $\mathcal{F}(as)$ are the
finite sets $\underline{m} = \{1, 2, 3, \ldots, m\}$ for $m \geq 1$.
A morphism $\lambda$ of $\mathcal{F}(as)$ is a set map
$\underline{\lambda} \colon\underline{m} \to \underline{n}$ together
with a specified total ordering $<_{\lambda}$ on each preimage set
$\underline{\lambda}^{-1}(i)$, $1 \leq i \leq n$.

Let $X = \{x_0, x_1, x_2, x_3, \ldots\}$ be a set of formal
indeterminates, and consider the free monoid, $X^{\star}$, generated
by $X$. Define the {\it tensor representation} of a morphism $f \in
\Delta S_+([n], [m])$ as the image of $(x_0, x_1, \ldots, x_n)$ under
$B^{sym_+}X^{\star}f$.  Typically, a morphism whose tensor
representation is $(y_0, y_1, \ldots, y_m)$ (in which each $y_i$ is a
possibly-empty monomial in the indeterminates $x_j$) will be written
$y_0 \otimes y_1 \otimes \cdots \otimes y_m$, hence the terminology.
The correspondence sending a morphism to its tensor representation is
one-to-one by uniqueness of decomposition of $\Delta S_+$ morphisms
into a $\Delta_+$ morphism, which determines the number of factors in
each monomial $y_i$, and a permutation, which determines the total
order of the indices.
\begin{example}
  Let $\phi \in \Delta_+([2], [1])$ be the map sending
  $i \mapsto i$ for $i=0,1$, and $2 \mapsto 1$.  Let $\gamma = (0,1,2)
  \in \Sigma_3$.  The tensor represenatation of $(\phi, \gamma)$ is
  $(x_1, x_2x_0)$, or $x_1 \otimes x_2x_0$.
\end{example}
Tensor notation provides the link to $\mathcal{F}(as)$.
\begin{proposition}\label{prp.iso}
  There is an isomorphism of categories $F \colon \Delta S \to
  \mathcal{F}(as)$.
\end{proposition}
\begin{proof}
  The functor $F$ takes $[n]$ to $\underline{n+1}$ for each $n \geq
  0$.  Let $f \colon [n] \to [m]$ be a morphism in $\Delta S$ and
  write $f = (y_0, y_1, \ldots, y_m)$ in tensor notation.  Then $F(f)
  = \lambda$, where $\underline{\lambda}$ is the set function such
  that $\underline{\lambda}(j) = i \; \Leftrightarrow \; x_{j-1}$
  appears as a factor in $y_{i-1}$, while the total ordering
  $<_{\lambda}$ on $\underline{\lambda}^{-1}(i)$ is induced by the
  ordering of factors in $y_{i-1}$, that is, if $y_{i-1} =
  x_{j_1-1}x_{j_2-1} \cdots x_{j_k-1}$, then $j_1 <_{\lambda} < j_2
  <_{\lambda} \cdots <_{\lambda} j_k$.  Bijectivity of $F$ is clear,
  and verifying that $F$ is indeed a functor is left to the reader.
\end{proof}
\begin{remark}
  If we denote by $\mathcal{F}(as)_+$ the category $\mathcal{F}(as)$
  enlarged by the initial object $\underline{0} = \emptyset$, then
  Prop.~\ref{prp.iso} implies $\Delta S_+ \cong \mathcal{F}(as)_+$.
\end{remark}
There are tantalizing links among symmetric homology, cyclic homology
and the so-called $\Gamma$-homology theories of Alan Robinson and
Sarah Ann Whitehouse and related $\Gamma(as)$ and $\mathcal{F}(as)$
homologies, theories that have been much studied
recently~\cite{Robinson,RobinsonWhitehouse,Whitehouse,Richter,PR}.
Pirashvili and Richter~\cite{PR} identify the cyclic homology of any
$\mathcal{F}(as)$-module $G$ with $\mathrm{Tor}_*^{\mathcal{F}(as)}(b,
G)$, where $b$ is the cokernel of a certain map of
$\mathcal{F}(as)^{\mathrm{op}}$-modules.  We shall interpret this
statement using $\Delta S$ presently.  Define for each $m \geq 0$ the
projective $\Delta S^{\mathrm{op}}$-module,
\begin{equation}\label{eqn.P_n-def}
  (\Delta S)_m \stackrel{def}{=} k\left[\Delta S(-, [m])\right],
\end{equation}
so in particular, for any $n \geq 0$, $(\Delta S)_m([n])$ is the free
$k$-module generated by the set $\Delta S([n], [m])$.  The covariant
version $(\Delta S)^m$ is defined analogously, however we have no need
for it in this paper.  In light of Prop.~\ref{prp.iso}, we may
interpret Pirashvili and Richter's result thus: $HC_*(G) \cong
\mathrm{Tor}_*^{\Delta S}(b, G)$, where $b$ fits into the exact
sequence below.
\begin{equation}
  \begin{tikzcd}
    (\Delta S)_1 \rar{\eta} \rar & (\Delta S)_0
    \rar & b \rar & 0,
  \end{tikzcd}
\end{equation}
and $\eta$ is defined on morphisms $f \colon [m] \to [1]$ by $\eta(f)
= x_0x_1 \circ f - x_1x_0 \circ f$.  When $G = B_*^{sym}A$, one finds
the {\it cyclic homology of the symmetric bar construction},
$HC_*(B_*^{sym}A)$, which coincides with the cyclic homology of $A$, as
Loday's {\it cyclic bar construction}~\cite{L}, $B_*^{cyc}A$, is the
restriction of $B_*^{sym}A$ under the inclusion of categories, $\Delta
C \hookrightarrow \Delta S$, and the duality isomorphism, $\Delta
C^{\mathrm{op}} \cong \Delta C$.  Indeed, we have a chain of
isomorphisms,
\begin{equation}
  HC_*(A) = HC_*(B^{cyc}A) \cong HC_*(B_*^{sym}A) \cong
  \mathrm{Tor}_*^{\Delta S}(b, B_*^{sym}A).
\end{equation}

\subsection{Homotopy-Everything Operads}\label{sub.he-operads}

Let $\mathbf{S}$ be the {\it symmetric groupoid}, which has as objects
$\underline{n}$ for $n \geq 0$, and whose only morphisms are the
permutations $\sigma \colon \underline{n} \to \underline{n}$.  Thus
$\mathbf{S}^{\mathrm{op}} \cong \mathrm{Aut}\Delta S_+$, via the map
$\sigma \mapsto (\mathrm{id}, \sigma)$.  Here,
$\mathrm{Aut}\mathscr{C}$ is the subcategory of $\mathscr{C}$
containing the same objects and only the automorphisms of
$\mathscr{C}$.  Therefore, any $\Delta S_+$ object is naturally an
$\mathbf{S}^{\mathrm{op}}$ object.  Present in the early work of
Boardman and Vogt, and developed later by May and others, is the
concept of {\it homotopy-everything}, or $E_{\infty}$,
operad~\cite{BV,M,MSS}.  As our operads will be defined in various
categories, not just topological spaces, it is important to clearly
define certain concepts.  Let $\mathscr{C}$ be a small symmetric
monoidal category with unit object $\mathbf{1}$.  Suppose there is a
model structure~\cite{Quillen67,Hovey2007} on $\mathscr{C}$ (although,
we only need the notion of equivalences, not (co)fibrations).  We
define an $E_{\infty}$ operad in $\mathscr{C}$ to be a functor
$\mathscr{P} \colon\mathbf{S}^{\mathrm{op}} \to \mathscr{C}$, with
structure maps satisfying the standard commutative diagrams of an
operad, such that each component $\mathscr{P}(n)$ is equivalent (in
the model structure) to $\mathbf{1}$.  We also require the symmetric
group action on each $\mathscr{P}(n)$ to be free.  In this note, we
are primarily interested in operads in the category of small
categories ($\mathbf{Cat}$), whose model structure is induced by the
nerve functor, and in simplicial sets
($\mathbf{Set}^{\Delta^{\mathrm{op}}}$), simplicial $k$-modules
($k\textrm{-}\mathbf{Mod}^{\Delta^{\mathrm{op}}}$), and
non-negatively-graded $k$-complexes ($\mathbf{Ch}^+_{\bullet}$) --
each with the standard model structure.

\begin{example}
  May's {\it little $\infty$-cubes operad} $\mathcal{C}_{\infty}$ is
  $E_{\infty}$ in the category of topological spaces.
\end{example}

\begin{example}\label{exa.D_Cat}
  Let $\mathscr{D}_{\mathbf{Cat}}$ denote the operad in $\mathbf{Cat}$
  defined by $\mathscr{D}_{\mathbf{Cat}}(m) = E\Sigma_m$. That is, the
  objects of $\mathscr{D}_{\mathbf{Cat}}(m)$ are the elements of the
  symmetric group on $m$ letters, and for each pair of objects
  $(\sigma, \tau)$, there is a unique morphism $\tau\sigma^{-1}$ from
  $\sigma$ to $\tau$. The structure map in
  $\mathscr{D}_{\mathbf{Cat}}$ is the family of functors
  $\mathscr{D}_{\mathbf{Cat}}(m) \times \mathscr{D}_{\mathbf{Cat}}
  (k_1) \times \cdots \times \mathscr{D}_{\mathbf{Cat}}(k_m)
  \longrightarrow \mathscr{D}_{\mathbf{Cat}}(k)$, where $k = \sum_i
  k_i$, defined on objects by:
  \begin{equation}
    (\sigma, \tau_1, \ldots, \tau_m) \mapsto \sigma_{\mathbf{k}}
    \cdot(\tau_1 \oplus \cdots \oplus \tau_m).
  \end{equation}
  The action of $\Sigma_m^{\mathrm{op}}$ on objects of
  $\mathscr{D}_{\mathbf{Cat}}(m)$ is given by right multiplication of
  group elements.  Since each $E\Sigma_m$ has free $\Sigma_m$ action
  and is a contractible category, the operad is $E_{\infty}$.
\end{example}

\begin{remark}
  The notation $\mathscr{D}_{\mathbf{Cat}}$ is related to the notation
  used in May~\cite{M,M3}.  May uses $\widetilde{\Sigma}_m$ for
  $\mathscr{D}_{\mathbf{Cat}}(m)$ and defines the related operad
  $\mathscr{D}$ in the category of spaces, as the geometric
  realization of the nerve of $\widetilde{\Sigma}$.  The nerve of
  $\mathscr{D}_{\mathbf{Cat}}$ is generally known in the literature as
  the {\it Barratt-Eccles operad} (See~\cite{BE}, where the notation
  for $N\mathscr{D}_{\mathbf{Cat}}$ is $\Gamma$, not to be confused
  with the $\Gamma$ of $\Gamma$-homology!).  We denote by
  $\mathscr{D}_{\mathbf{Mod}}$ the associated $E_{\infty}$ operad in
  the category of simplicial $k$-modules defined by
  $\mathscr{D}_{\mathbf{Mod}}(m) = E_*\Sigma_m$ (the standard bar
  resolution of $k$ by free $k[\Sigma_m]$-modules), and the Moore
  complex (that is, the complex of normalized chains~\cite{GJ}) of
  $\mathscr{D}_{\mathbf{Mod}}(m)$ by
  $\mathscr{D}_{\mathbf{Ch}_{\bullet}^{+}}(m)$.
\end{remark}

\subsection{Operad-algebras}\label{sub.operad-algebras}

By {\it operad-algebra}, we mean an algebra over an operad in the
usual sense (as in~\cite{MSS}: II.1.4), in which the algebra lies in
the same underlying category as the operad acting on it.  As an
example, if $\mathscr{C}$ is a permutative category, then
$B\mathscr{C}$ is naturally an $E_{\infty}$-space~\cite{M3} (that is,
an $E_{\infty}$ algebra in $\mathbf{Top}$).  In fact, $\mathscr{C}$ is
itself an $E_{\infty}$ algebra in $\mathbf{Cat}$.  It is useful to
regard a permutative $\mathscr{C}$ explicitly as
$\mathscr{D}_{\mathbf{Cat}}$-algebra according to the structure map
$\theta$ of Diagram~(\ref{eq.theta-diagram}).  Here, $f_i \colon C_i
\to D_i$ for each $i = 1, 2, \ldots, m$, and the map $T_{\tau
  \sigma^{-1}}$ permutes the components according to the permutation
$\tau \sigma^{-1}$ using the symmetry transformation and strict
associativity of the monoidal product $\odot$ of $\mathscr{C}$.
\begin{equation}\label{eq.theta-diagram}
  \begin{tikzcd}
    (\sigma, C_1, \ldots, C_m) \rar{\theta} \ar{dd}[swap]{\tau
      \sigma^{-1} \times \mathbf{f} }& \sigma\mathbf{C}^{\odot}
    \dar{\sigma\mathbf{f}^{\odot}}\\ & \sigma\mathbf{D}^{\odot}
    \dar[swap]{ \cong }[swap]{ T_{\tau\sigma^{-1}} }\\ (\tau, D_1,
    \dots, D_m) \rar{\theta} & \tau\mathbf{D}^{\odot}
  \end{tikzcd}
\end{equation}

\section{Operad Structure within Symmetric Homology}

\subsection{Monoid Algebras}

In order to produce an $E_{\infty}$ structure for the simplicial
module that computes symmetric homology, we first have to work at the
level of monoids and simplicial sets.
\begin{lemma}\label{lem.X-perm-cat}
  Let $M$ be a monoid.  $C(\Delta S_+, B^{sym_+}M)$ has the structure
  of $E_{\infty}$ algebra in the category of simplicial sets.
\end{lemma}
\begin{proof}
  Consider $\mathcal{T}M$, as in \S\ref{sub.monoids}. A typical
  $i$-simplex of $N\mathcal{T}M$ has the form,
  \begin{equation}\label{eqn.chain1}
    \langle f_i \cdots f_2f_1, \mathbf{m} \rangle \stackrel{ f_i
    }{\longleftarrow} \cdots \stackrel{ f_3 } {\longleftarrow} \langle
    f_2f_1, \mathbf{m}\rangle \stackrel{ f_2 } {\longleftarrow}
    \langle f_1, \mathbf{m}\rangle \stackrel{ f_1 }{\longleftarrow}
    \mathbf{m},
  \end{equation}
  in which $\mathbf{m} = (m_0, m_1, \ldots, m_n)$.
  Expression~(\ref{eqn.chain1}) can be rewritten uniquely as an
  element of $M^{n+1}$ together with an element ($i$-simplex) of
  $N\Delta S_+$:
  \begin{equation}\label{eqn.chain2}
    \left( [n_i] \stackrel{f_i}{\gets} \cdots \stackrel{f_3}{\gets}
         [n_2] \stackrel{f_2}{\gets} [n_1] \stackrel{f_1}{\gets} [n]
         \,,\, \mathbf{m} \right),
  \end{equation}
  which in turn is uniquely identified with an element of $C(\Delta
  S_+, B^{sym_+}M)$:
  \begin{equation}\label{eqn.chain3}
    \left( [n_i] \stackrel{f_i}{\gets} \cdots \stackrel{f_3}{\gets}
         [n_2] \stackrel{f_2}{\gets} [n_1] \stackrel{f_1}{\gets} [n]
         \stackrel{\mathrm{id}}{\gets} [n] \,,\, \mathbf{m} \right).
  \end{equation}
  Thus,~(\ref{eqn.chain1})--(\ref{eqn.chain3}) define a map $L_M
  \colon N\mathcal{T}M \to C(\Delta S_+, B^{sym_+}M)$.

  On the other hand, a typical element of $C(\Delta S_+, B^{sym_+}M)$
  may not {\it a priori} have an identity morphism $[n] \to [n]$ as
  the ``incoming morphism'', but by using $\Delta S_+$-equivariance,
  we can always express the element in the desired form:
  \begin{equation}\label{eqn.chain4}
    \left( [n_i] \stackrel{f_i}{\gets} \cdots \stackrel{f_2}{\gets}
         [n_1] \stackrel{f_1}{\gets} [n] \stackrel{f_0}{\gets} [n']
         \,,\, \mathbf{m}' \right) = \left( [n_i]
         \stackrel{f_i}{\gets} \cdots \stackrel{f_2}{\gets} [n_1]
         \stackrel{f_1}{\gets} [n] \stackrel{\mathrm{id}}{\gets} [n]
         \,,\, \langle f_0, \mathbf{m}'\rangle \right).
  \end{equation}
  This element is identified with the following $i$-simplex of
  $N\mathcal{T}M$:
  \begin{equation}\label{eqn.chain5}
    \langle f_i \cdots f_0, \mathbf{m}' \rangle \stackrel{ f_i
    }{\longleftarrow} \cdots \stackrel{ f_2 } {\longleftarrow} \langle
    f_1f_0, \mathbf{m}'\rangle \stackrel{ f_1 }{\longleftarrow}
    \langle f_0, \mathbf{m}' \rangle.
  \end{equation}
  Thus,~(\ref{eqn.chain4})--(\ref{eqn.chain5}) define a map $R_M
  \colon C(\Delta S_+, B^{sym_+}M) \to N\mathcal{T}M$.

  Clearly, $L_M$ and $R_M$ are simplicial maps that are inverse of one
  another and the isomorphism follows:
  \begin{equation}
    C(\Delta S_+, B^{sym_+}M) \cong N\mathcal{T}M.
  \end{equation}
  
  By Lemma~\ref{lem.TM_permutative}, $\mathcal{T}M$ is permutative.
  Since the nerve functor $N$ is symmetric monoidal, the
  $\mathscr{D}_{\mathbf{Cat}}$-algebra structure of
  Diagram~\ref{eq.theta-diagram} is induced to the level of simplicial
  sets.  This implies $N\mathcal{T}M$, and hence also $C(\Delta S_+,
  B^{sym_+}M)$, is an $E_{\infty}$ algebra.
\end{proof}

\begin{remark}
  The fact that neither $\Delta_+$ nor $\Delta C_+$ are permutative
  categories implies that the proof of Lemma~\ref{lem.X-perm-cat} does
  not extend to simplicial or cyclic homology.  However it is
  interesting to see that in certain special cases, there does seem to
  be way to define a Dyer-Lashoff structure on cyclic
  homology~\cite{BergRog}.
\end{remark}

\begin{remark}
  The proof of Lemma~\ref{lem.X-perm-cat} unfortunately does not
  extend directly to arbitrary algebras.  Indeed this is a serious
  obstruction to Theorem~8 of~\cite{AF}!  Presently, we do not have a
  way to prove that $HS_*(A) \cong H_*(B(D, T, A))$ (as~\cite{AF}
  claims), where $D$ is the monad associated to the operad
  $\mathscr{D}_{\mathbf{Mod}}$, and $T$ is the functor that takes a
  $k$-module to its tensor algebra.  The author suspects that the
  isomorphism is false for arbitrary algebras.
\end{remark}

\subsection{A Structure Map}\label{sub.structure_map}

For each $m \geq 0$, set
\begin{equation}
  \label{eqn.F(m)}
  \mathcal{F}(m) \stackrel{def}{=} [m-1] \setminus \Delta S_+.
\end{equation}
Eqn.~(\ref{eqn.F(m)}) defines $\mathcal{F}$ as a $\Delta
S_+^{\mathrm{op}}$ category via $[n] \mapsto \mathcal{F}(n+1)$, hence
also as an $\mathbf{S}$ category.  Precompose the duality functor
$\mathbf{S}^{\mathrm{op}} \to \mathbf{S}$ sending $\sigma \mapsto
\sigma^{-1}$ to define $\mathcal{F}$ as a $\mathbf{S}^{\mathrm{op}}$
category.  Because of the many ``reversals'' wrapped up in this
definition, it is important to show the details.  For each $m \geq 0$,
there is a {\it right} $\Sigma_m$ action on $\mathcal{F}(m)$,
\begin{eqnarray}
  \label{eqn.right_action1}
  (\phi, \gamma) \bullet \sigma &\stackrel{def}{=}& (\phi, \gamma)
  \circ (\mathrm{id}, \sigma^{-1})\\
  \label{eqn.right_action2}
  &=& (\phi, \sigma^{-1}\gamma).
\end{eqnarray}
There is also a left $\Sigma_n$ action on each set
$\Delta S_+\left( [m-1], [n-1] \right)$:
\begin{eqnarray}
  \label{eqn.left_action1}
  \tau \bullet (\phi, \gamma) &\stackrel{def}{=}& (\mathrm{id}, \tau^{-1})
  \circ (\phi, \gamma) \\ 
  \label{eqn.left_action2}
  &=& (\phi^{(\tau^{-1})},
  \gamma(\tau^{-1})^{\phi}),
\end{eqnarray}
and the two actions commute in the sense that
\begin{equation}
  \tau \bullet (f \bullet \sigma) = (\tau \bullet f) \bullet \sigma.
\end{equation}

Let $m, j_1, j_2, \ldots, j_m \geq 0$ and $j = \sum j_s$.  Assume
morphisms $f_i, g_i$ of $\Delta S_+$, for $1 \leq i \leq
m$, have specified sources and targets: $[j_i-1] \stackrel{f_i}{\to}
[p_i-1] \stackrel{g_i}{\to} [q_i-1]$.  Define a family of maps,
\begin{equation}
  \mu = \mu_{m,j_1,\ldots,j_m} \colon\mathscr{D}_{\mathbf{Cat}}(m)
  \times \prod_{s=1}^m \mathcal{F}(j_s) \longrightarrow
  \mathcal{F}(j),
\end{equation}
on objects by
\begin{equation}\label{eq.mu-objects}
  \mu( \sigma, f_1, f_2, \ldots, f_m ) \stackrel{def}{=}
  \sigma_{\mathbf{p}} \bullet \mathbf{f}^{\odot}.
\end{equation}
Define $\mu$ on morphisms by Diagram~(\ref{eq.mu-morphisms}).
\begin{equation}\label{eq.mu-morphisms}
  \begin{tikzcd}[column sep=large]
    (\sigma, f_1, \ldots, f_m)\arrow{ddd}[swap]{\tau\sigma^{-1} \times
      \mathbf{g} } \rar{\mu} & \sigma_{\mathbf{p}} \bullet
    \mathbf{f}^{\odot} \dar{(\sigma_{\mathbf{p}})^{-1}}\\ &
    \mathbf{f}^{\odot} \dar{\mathbf{g}^{\odot}} \\ &
    \mathbf{gf}^{\odot} \dar{ \tau_{\mathbf{q}} } \\ (\tau, g_1f_1,
    \ldots, g_mf_m) \rar{\mu} & \tau_{\mathbf{q}} \bullet
    \mathbf{gf}^{\odot}
  \end{tikzcd}
\end{equation}
The effect is simply ``untwisting'' by block permutation, applying the
morphisms $g_i$ in the natural order, then ``retwisting'' by the
appropriate block permutation.  Functoriality of $\mu$ is clear.  We
show in \S \ref{sub.operad-module-structure} that the maps $\mu$
define a left operad-module structure (over
$\mathscr{D}_{\mathbf{Cat}}$) on $\mathcal{F}$ (the reader is referred
to~\cite{MSS} for the definition of operad-module).

\subsection{Operad-module Structure of $\mathcal{F}$.}
\label{sub.operad-module-structure}

Consider the set of formal indeterminates $X = \{x_1, x_2, \ldots\}$,
and the free monoid $X^{\star}$ as defined in
\S\ref{sub.noncomm_sets}.  In this section, we prove that
$\mathcal{T}X^{\star}$ is isomorphic, as a category, to a certain
category built from $\mathcal{F}$.  We then use this isomorphism to
prove that $\mathcal{F}$ admits the structure of operad-module over
$\mathscr{D}_{\mathbf{Cat}}$.

Since $\mathcal{F}$ is an $\mathbf{S}^{\mathrm{op}}$ object, there is
a right action of the symmetric group $\Sigma_m$ on $\mathcal{F}(m)$
for each $m \geq 0$; recall
Eqns~(\ref{eqn.right_action1})--(\ref{eqn.right_action2}).  There is a
left action of $\Sigma_m$ on $X^m$ by permutation, $\sigma \bullet
\mathbf{x} = \sigma\mathbf{x} = \left(x_{\sigma^{-1}(1)}, \ldots,
x_{\sigma^{-1}(m)}\right)$.  Thus, the fibered product of categories
can be formed, $\mathcal{F}(m) \times_{\Sigma_m} X^m$.  Here $X$ is
taken to be a discrete category.  

For a set $\{ \mathscr{C}_i \}$ of small categories whose object sets
are pairwise disjoint, we use the notation $\bigcup_i \mathscr{C}_i$
to represent the category whose object set is $\bigcup_i
\mathrm{Ob}\mathscr{C}_i$, and whose morphisms only those morphisms in
$\mathrm{Mor}{\mathscr{C}_i}$ for each $i$.  This is, of course, a
particular realization of the coproduct of a set of small categories.
\begin{lemma}\label{lem.TX^star}
  There is an isomorphism of categories,
  \begin{equation}
    e \colon \bigcup_{m \geq 0}\left(\mathcal{F}(m) \times_{\Sigma_m}
    X^m\right) \to \mathcal{T}X^{\star},
  \end{equation}
  via the evaluation map $e$ defined by $e(f, \mathbf{x})
  \stackrel{def}{=} \langle f, \mathbf{x} \rangle$.
\end{lemma}
\begin{proof}
  We must show that the evaluation functor,
  \begin{eqnarray}
    \mathcal{F}(m) \times X^m &\to& \mathcal{T}X^{\star}\\ (f,
    \mathbf{x}) &\mapsto& \langle f, \mathbf{x} \rangle,
  \end{eqnarray} 
  factors through the canonical projection $\mathcal{F}(m) \times X^m
  \to \mathcal{F}(m) \times_{\Sigma_m} X^m$.  Let $f$ be a $\Delta
  S_+$ morphism and write $f = \phi \circ \gamma$ with $\phi$
  a morphism of $\Delta_+$ and $\gamma \in \Sigma_m^{\mathrm{op}}$.  By
  unique factorization in $\Delta S_+$, the pair is unique to $f$.
  Let $\mathbf{x} \in X^m$, and let $\sigma \in \Sigma_m$.  Observe
  that Property~(\ref{eqn.eval_prop}) is used to ``transfer'' the
  permutation from the left to the right.
  \begin{eqnarray}
    \langle f \bullet \sigma, \mathbf{x} \rangle &=& \langle \phi \circ
    \sigma^{-1}\gamma, \mathbf{x} \rangle \\
    &=& \big\langle \phi, \langle
    \sigma^{-1}\gamma, \mathbf{x}\rangle \big\rangle\\
    &=& \langle \phi,
    \gamma^{-1}\sigma\mathbf{x} \rangle\\
    &=& \big\langle \phi, \langle
    \gamma, \sigma\mathbf{x} \rangle \big\rangle\\
    &=& \langle \phi \circ
    \gamma, \sigma\mathbf{x} \rangle \\
    &=&
    \langle f, \sigma \bullet \mathbf{x} \rangle.
  \end{eqnarray}
  There is also a map (on objects) in the reverse direction, defined
  by
  \begin{equation}
    (y_1, y_2, \ldots, y_n) \mapsto (\phi, x_{i_1}, x_{i_2}, \ldots,
    x_{i_m}),
  \end{equation}
  where each $y_i$ is a possibly-empty monomial in the indeterminates
  $x_j$, such that $y_1y_2 \cdots y_n = x_{i_1}x_{i_2} \cdots x_{i_m}
  \in X^{\star}$, and $\phi$ is the $\Delta_+$ morphism such that
  $\phi(j-1) = j'-1 \Leftrightarrow x_{i_j}$ appears as a factor in
  $y_{j'}$.  Whereas the map $e$ has the effect of multiplying certain
  groups of indeterminates together, the reverse map {\it factors} the
  monomials completely, which can be done uniquely since $X^{\star}$
  is a free monoid.  The two maps are inverse to one another, making
  $e$ bijective on objects.

  We have yet to define $e$ on morphisms.  Observe that since $X^m$ is
  discrete, the morphisms of $\mathcal{F}(m) \times_{\Sigma_m} X^m$
  all have the form $g \times \mathrm{id}_{x_{i_1}} \times \cdots
  \times \mathrm{id}_{x_{i_m}}$.  The functor $e$ simply maps this
  morphism to $g$ as interpreted in $\mathcal{T}X^{\star}$, as the
  commutative diagram~(\ref{eqn.e_on_morphisms}) illustrates.
  \begin{equation}\label{eqn.e_on_morphisms}
    \begin{tikzcd}
      (f, \mathbf{x}) \rar{e} \dar[swap]{ g \times \mathrm{id}} &
      \langle f, \mathbf{x}\rangle \dar{ g } \\ (gf, \mathbf{x})
      \rar{e} & \langle gf, \mathbf{x}\rangle.
    \end{tikzcd}
  \end{equation}
  It is straightforward to check that $e$ is fully faithful, and so
  $e$ is an isomorphism of categories as claimed.
\end{proof}

For the remainder of this section, we prove that the family of maps
$\mu$ defined in \S\ref{sub.structure_map} give $\mathcal{F}$ the
structure of an operad-module over $\mathscr{D}_{\mathbf{Cat}}$.  Fix
integers $m, j_1, j_2, \cdots, j_m \geq 0$, and let $j = \sum j_s$.
In this section $\mathbf{x} = (x_1, x_2, \ldots, x_j)$.  We will need
to partition $\mathbf{x}$ into chunks of sizes $j_1, j_2, \ldots,
j_s$.  To that end, define for each $s$,
\begin{eqnarray*}
  \mathbf{x}_1 &=& (x_1, \ldots, x_{j_1}),\\ \mathbf{x}_s &=& (x_{j_1
    + \cdots + j_{s-1} + 1} , \ldots, x_{j_1 + \cdots + j_s}), \quad
  \textrm{for $s > 1$.}
\end{eqnarray*}
For each number $s = 1, 2, \ldots, m$, let $a_s$ be the inclusion
functor,
\begin{equation}
  a_s \colon \mathcal{F}(j_s) \longrightarrow \mathcal{F}(j_s)
  \times_{\Sigma_{j_s}} X^{j_s},
\end{equation}
\begin{equation}
  f \mapsto (f, \mathbf{x}_s).
\end{equation}
We also require a similar functor,
\begin{equation}
  a \colon \mathcal{F}(j) \longrightarrow \mathcal{F}(j)
  \times_{\Sigma_{j}} X^{j},
\end{equation}
\begin{equation}
  f \mapsto \left(f, \mathbf{x}\right).
\end{equation}
Consider the functor (in which $\mathbf{a} = a_1 \times \ldots \times
a_m$),
\begin{equation}\label{eqn.alpha_tilde}
  \tilde{a} \stackrel{def}{=} \mathrm{id} \times \mathbf{a} \colon
  \mathscr{D}_{\mathbf{Cat}}(m) \times \prod_{s=1}^m \mathcal{F}(j_s)
  \longrightarrow \mathscr{D}_{\mathbf{Cat}} (m) \times \prod_{s=1}^m
  \left(\mathcal{F}(j_s) \times_{\Sigma_{j_s}} X^{j_s} \right).
\end{equation}
For any number $i \geq 0$, let $b_i$ be the inclusion of
categories:
\begin{equation}
  b_i \colon \mathcal{F}(i) \times_{\Sigma_{i}} X^{i}
  \longrightarrow \bigcup_{i \geq 0} \mathcal{F}(i) \times_{\Sigma_i}
  X^i,
\end{equation}
Define $\tilde{b}$ analogously to~(\ref{eqn.alpha_tilde}):
\begin{equation}
  \tilde{b} \stackrel{def}{=} \mathrm{id} \times \mathbf{b} \colon
  \mathscr{D}_{\mathbf{Cat}}(m) \times \prod_{s=1}^m
  \left(\mathcal{F}(j_s) \times_{\Sigma_{j_s}} X^{j_s} \right)
  \longrightarrow \mathscr{D}_{\mathbf{Cat}}(m) \times \left[
    \bigcup_{i \geq 0} \left(\mathcal{F}(i) \times_{\Sigma_i} X^i
    \right)\right]^m.
\end{equation}
Now, by Lemma~\ref{lem.TX^star}, there is an isomorphism,
\begin{equation}
  \tilde{e} \stackrel{def}{=} \mathrm{id} \times e^m \colon
  \mathscr{D}_{\mathbf{Cat}}(m) \times \left[ \bigcup_{i \geq 0}
    \left( \mathcal{F}(i) \times_{\Sigma_i} X^i \right)\right]^m
  \stackrel{\cong}{\longrightarrow} \mathscr{D}_{\mathbf{Cat}}(m)
  \times \left( \mathcal{T}X^{\star}\right)^m.
\end{equation}
Consider the following diagram.  The top row is the map $\mu$ of
Eq.~(\ref{eq.mu-objects}), and the bottom row is the operad-algebra
structure map for $\mathcal{T}X^{\star}$, which comes from the
$\mathscr{D}_{\mathbf{Cat}}$-algebra structure of this permutative
category (see Lemma~\ref{lem.TM_permutative}).
\begin{equation}\label{eq.operad--module-comm-diag}
  \begin{tikzcd}
    \ds{\mathscr{D}_{\mathbf{Cat}}(m) \times \prod_{s=1}^m
      \mathcal{F}(j_s)} \rar{\mu} \dar{\tilde{a}} & \mathcal{F}(j)
    \dar{a} \\ \ds{\mathscr{D}_{\mathbf{Cat}}(m) \times \prod_{s=1}^m
      \left(\mathcal{F}(j_s)\times_{\Sigma_{j_s}} X^{j_s} \right)}
    \dar{\tilde{b}} & \mathcal{F}(j) \times_{\Sigma_{j}} X^{j} \dar{b_j}
    \\ \ds{\mathscr{D}_{\mathbf{Cat}}(m) \times \left[ \bigcup_{i \geq
          0} \left( \mathcal{F}(i) \times_{\Sigma_i} X^i
        \right)\right]^m} \dar{\tilde{e}}& \ds{ \bigcup_{i \geq 0}
      \left(\mathcal{F}(i) \times_{\Sigma_i} X^i\right)}
    \dar{e}\\ \mathscr{D}_{\mathbf{Cat}}(m) \times \left(
    \mathcal{T}X^{\star}\right)^m \rar{\theta} & \mathcal{T}X^{\star}
  \end{tikzcd}
\end{equation}
Diagram~(\ref{eq.operad--module-comm-diag}) commutes if we can show
that $\theta\tilde{e}\tilde{b}\tilde{a} = eb_ja\mu$.  Let $w = (\sigma,
f_1, \ldots, f_m) \in \mathscr{D}_{\mathbf{Cat}}(m) \times
\prod_{s=1}^m \mathcal{F}(j_s)$ be arbitrary.  First follow the
element $w$ down the left column and across the bottom of
Diagram~(\ref{eq.operad--module-comm-diag}).
\begin{equation}\label{eq.left_column_on_w}
  \begin{tikzcd}
    (\sigma, f_1, \ldots, f_m) \dar{\tilde{b}\tilde{a}}
    \\ \left(\sigma, (f_1, \mathbf{x}_1), \ldots, (f_m, \mathbf{x}_m)\right)
    \dar{\tilde{e}}\\ \left( \sigma, \langle f_1, \mathbf{x}_1 \rangle,
    \ldots, \langle f_m, \mathbf{x}_m \rangle \right) \rar{\theta} & \langle
    f_{\sigma^{-1}(1)}, \mathbf{x}_{\sigma^{-1}(1)} \rangle \odot \cdots
    \odot \langle f_{\sigma^{-1}(m)}, \mathbf{x}_{\sigma^{-1}(m)}\rangle.
  \end{tikzcd}
\end{equation}
Now follow the element $w$ across the top and down the right column of
Diagram~(\ref{eq.operad--module-comm-diag}).  Assume the codomain of
$f_i$ is $[p_i-1]$ for each $i \leq m$.
\begin{equation}\label{eq.right_column_on_w}
  \begin{tikzcd}
    (\sigma, f_1, \ldots, f_m) \rar{\mu} & \sigma_{\mathbf{p}} \bullet
    \mathbf{f}^{\odot} \dar{b_ja} \\ &\left( \sigma_{\mathbf{p}} \bullet
    \mathbf{f}^{\odot}, \mathbf{x} \right) \dar{e}\\ & \langle
    \sigma_{\mathbf{p}} \bullet \mathbf{f}^{\odot}, \mathbf{x}
    \rangle.
  \end{tikzcd}
\end{equation}
Now since $\mathbf{x} = (\mathbf{x}_1, \ldots, \mathbf{x}_m)$, the
bottom right element in Diagram~(\ref{eq.right_column_on_w}) may be
simplified thus:
\begin{eqnarray}
  \langle \sigma_{\mathbf{p}} \bullet \mathbf{f}^{\odot}, \mathbf{x}
  \rangle &=& \big\langle (\mathrm{id}, \sigma_{\mathbf{p}}^{-1}) \circ
  \mathbf{f}^{\odot}, (\mathbf{x}_1, \ldots, \mathbf{x}_m)
  \big\rangle\\ &=& \big\langle (\mathrm{id}, \sigma_{\mathbf{p}}^{-1})
  \langle \mathbf{f}^{\odot}, (\mathbf{x}_1, \ldots,
  \mathbf{x}_m)\rangle \big\rangle\\ &=& \big\langle (\mathrm{id},
  \sigma_{\mathbf{p}}^{-1}), \langle f_1, \mathbf{x}_1 \rangle \odot \cdots
  \odot \langle f_m, \mathbf{x}_m\rangle \big\rangle\\ &=& \langle
  f_{\sigma^{-1}(1)}, \mathbf{x}_{\sigma^{-1}(1)} \rangle \odot \cdots
  \odot \langle f_{\sigma^{-1}(m)},
  \mathbf{x}_{\sigma^{-1}(m)}\rangle.
\end{eqnarray}
Using Diagram~(\ref{eq.operad--module-comm-diag}), we find that $\mu$
is an operad-module structure map.  Associativity is induced by the
associativity condition of the algebra structure map $\theta$ (because
both $eb_ja$ and $\tilde{e}\tilde{b}\tilde{a}$ are injective).  It is
trivial to verify the left unit condition (note, there is no
corresponding right unit condition in an operad-module structure).  We
include the routine check that verifies the equivariance condition on
the level of objects.  Assume $f_i \in \mathcal{F}(j_i)$ (for $1 \leq
i \leq m$) have specified sources and targets, $[j_i-1]
\stackrel{f_i}{\to} [p_i-1]$.  Recall, the symmetric group acts on the
         {\it right}.

{\bf Equivariance A:}
\begin{equation}\label{eqn.equiv_A}
  \begin{tikzcd}
    (\sigma, \mathbf{f}) \rar[mapsto]{ \mathrm{id} \times T_{\tau}
    } \dar[mapsto]{ \tau \times \mathrm{id} } & (\sigma,
    \tau\mathbf{f}) \dar[mapsto]{\mu}\\ (\sigma\tau, \mathbf{f})
    \dar[mapsto]{\mu} & \sigma_{\tau\mathbf{p}} \bullet
    \tau\mathbf{f}^{\odot} \dar[mapsto]{ \tau_{\mathbf{j}} }
    \\ (\sigma\tau)_{\mathbf{p}} \bullet \mathbf{f}^\odot \rar[equals]
    & \sigma_{\tau\mathbf{p}} \bullet \tau\mathbf{f}^{\odot} \bullet
    \tau_{\mathbf{j}}
  \end{tikzcd}
\end{equation}

{\bf Equivariance B:}
\begin{equation}\label{eqn.equiv_B}
  \begin{tikzcd}
    (\sigma, \mathbf{f}) \rar[mapsto]{\mu}
    \dar[mapsto][swap]{\mathrm{id} \times \tau_1 \times \cdots \times
      \tau_m} & \sigma_{\mathbf{p}} \bullet \mathbf{f}^\odot
    \arrow[mapsto]{dd}{\tau_1 \oplus \cdots \oplus \tau_m} \\ (\sigma,
    f_1 \bullet \tau_1, \ldots, f_m \bullet \tau_m) \dar[mapsto]{\mu}
    & \\ \sigma_{\mathbf{p}} \bullet \left( (f_1 \bullet \tau_1) \odot
    \cdots \odot (f_m \bullet \tau_m)\right) \rar[equals]&
    \sigma_{\mathbf{p}} \bullet \mathbf{f}^{\odot} \bullet
    \left(\tau_1 \oplus \cdots \oplus \tau_m\right)
  \end{tikzcd}
\end{equation}

\begin{remark}
  It can be verified that $\mathcal{F}$ is in fact a pseudo-operad.
  The details are left to the reader, as this result will not be used
  in the present paper.  Recall from~\cite{MSS} that a pseudo-operad
  is a `non-unitary' operad.  The structure maps are defined by the
  composition:
  \begin{equation}
    \begin{tikzcd}
      \ds{\mathcal{F}(m) \times \prod_{s=1}^m \mathcal{F}(j_s)}
      \rar{\pi \times \mathrm{id}} & \ds{\mathscr{D}_{\mathbf{Cat}}(m)
        \times \prod_{s=1}^m \mathcal{F}(j_s)} \rar{\mu} &
      \mathcal{F}(j_1 + \cdots + j_m),
    \end{tikzcd}
  \end{equation}
  where $\pi \colon\mathcal{F}(m) \to \mathscr{D}_{\mathbf{Cat}}(m)$
  is the projection functor defined by $\pi(\phi, \gamma) =
  \gamma^{-1}$.  Indeed, $\pi$ defines an isomorphism of the
  subcategory $\mathrm{Aut}\left([m-1] \setminus \Delta S_+\right)$
  onto $\mathscr{D}_{\mathbf{Cat}}(m)$.  Note, $\mathcal{F}$ is not a
  full operad, since it fails the right-unit condition.
\end{remark}

We shall denote the associated simplicial $k$-module
$\widetilde{\mathcal{F}} \stackrel{def}{=} k[N(- \setminus \Delta
  S_+)]$.
\begin{corollary}\label{cor.K_ch}
  There is a $\mathscr{D}_{\mathbf{Mod}}$-module structure on
  $\widetilde{\mathcal{F}}$.
\end{corollary}
\begin{proof}
  The $\mathscr{D}_{\mathbf{Cat}}$-module structure of $\mathcal{F}$
  gets induced via the chain of symmetric monoidal functors,
  \begin{equation}
  \begin{tikzcd}
    \mathbf{Cat} \rar{N} & \mathbf{Set}^{\Delta^{\mathrm{op}}} \rar{k[
        - ]} & k\textrm{-}\mathbf{Mod}^{\Delta^{\mathrm{op}}}.
  \end{tikzcd}
  \end{equation}
\end{proof}

\subsection{Operad-algebra Structure}\label{sec.alg-structure} 

In this subsection we use the operad-module structure defined in \S
\ref{sub.operad-module-structure} to induce a related operad-algebra
structure.  Let us first recall a fact of operad theory:
\begin{proposition}\label{prop.operad--algebra}
  Suppose $(\mathscr{C}, \oplus, \odot)$ is a cocomplete distributive
  symmetric monoidal category, $\mathscr{P}$ is an operad in
  $\mathscr{C}$, $\mathscr{L}$ is a left $\mathscr{P}$-module, and $Z
  \in \mathrm{Obj}\mathscr{C}$.  Then
  \begin{equation}
    \mathscr{L} \langle Z \rangle \stackrel{def}{=} \bigoplus_{m \geq 0}
    \mathscr{L}(m) \odot_{\Sigma_m} Z^{\odot m}
  \end{equation}
  admits the structure of a $\mathscr{P}$-algebra.
\end{proposition}
\begin{remark}
  The notation $\mathscr{L} \langle Z \rangle$ appears in Kapranov and
  Manin~\cite{KM} (where they use it in the category of vector
  spaces).  The concept is also present in~\cite{MSS} as the {\it
    Schur functor} of an operad (\cite{MSS}, Def~1.24).
\end{remark}

\begin{lemma}\label{lem.E_infty-algebra}
  The simplicial $k$-module $\widetilde{\mathcal{F}}
  \otimes_{\mathrm{Aut}\Delta S_+} B^{sym_+}_*A$ admits
  the structure of an $E_\infty$ algebra.
\end{lemma}
\begin{proof}
  One may identify: 
  \begin{equation}
    \widetilde{\mathcal{F}} \otimes_{\mathrm{Aut}\Delta S_+}
    B^{sym_+}_*A = \bigoplus_{n \geq 0}
    \widetilde{\mathcal{F}}(n) \otimes_{\Sigma_n} A^{\otimes n} =
    \widetilde{\mathcal{F}} \langle A \rangle.
  \end{equation}
  The result then follows directly from Cor.~\ref{cor.K_ch} and
  Prop.~\ref{prop.operad--algebra}.
\end{proof}

In what follows, denote $CA_* \stackrel{def}{=} C_*(\Delta S_+,
B^{sym_+}_*A)$, the simplicial $k$-module defined in
Eqn.~(\ref{eqn.GZ-chain}).  Note that
\begin{equation}
  CA_* = \widetilde{\mathcal{F}} \otimes_{\Delta S_+} B^{sym_+}_*A.
\end{equation}
The inclusion $\mathrm{Aut}\Delta S_+ \hookrightarrow \Delta S_+$
induces a quotient map $Q \colon \widetilde{F}\langle A \rangle \to
CA_*$.

\begin{lemma}\label{lem.E_infty-algebra_k-modules}
  The $\mathscr{D}_{\mathbf{Mod}}$-algebra structure on
  $\widetilde{\mathcal{F}} \otimes_{\mathrm{Aut}\Delta S_+}
  B^{sym_+}_*A$ induces a $\mathscr{D}_{\mathbf{Mod}}$-algebra
  structure on $CA_*$, which implies that $CA_*$ is an $E_{\infty}$
  algebra in the category of simplicial $k$-modules.
\end{lemma}
\begin{proof}
  Let $\nu$ be the structure map implied by
  Lemma~\ref{lem.E_infty-algebra} (which is ultimately induced by the
  structure map $\mu$ of \S \ref{sub.structure_map}):
  \begin{equation}  
    \nu \colon\mathscr{D}_{\mathbf{Mod}}(n) \otimes_{\Sigma_n}
    \left(\widetilde{\mathcal{F}}\langle A \rangle\right)^{\otimes n}
    \longrightarrow \widetilde{\mathcal{F}}\langle A \rangle.
  \end{equation}
  We will show that $\nu$ remains well-defined upon passing to the
  quotient, as illustrated in Diagram~(\ref{eqn.passing_to_quotient}).
  \begin{equation}\label{eqn.passing_to_quotient}
  \begin{tikzcd}
    \mathscr{D}_{\mathbf{Mod}}(n) \otimes_{\Sigma_n}
    \widetilde{\mathcal{F}}\langle A \rangle^{\otimes n} \rar{\nu}
    \dar{\mathrm{id} \otimes Q^{\otimes n}} & \widetilde{\mathcal{F}}
    \langle A \rangle \dar{Q} \\ \mathscr{D}_{\mathbf{Mod}}(n)
    \otimes_{\Sigma_n} \left( CA_* \right)^{\otimes n} \rar{\nu} &CA_*
  \end{tikzcd}
  \end{equation}

  It suffices to check that the structure is well-defined in degree
  $0$, because the face and degeneracy maps are induced by
  compositions and evaluations of $\Delta S_+$ morphisms.  A generator
  of $\mathscr{D}_{\mathbf{Mod}}(n) \otimes_{\Sigma_n} (CA_*)^{\otimes
    n}$ in degree $0$ has the following form:
  \begin{equation}\label{eq.D--algebra-invarianceL}
    \sigma \otimes \left( g_1f_1 \otimes V_1 \right) \otimes \cdots
    \otimes \left( g_nf_n \otimes V_n \right),
  \end{equation}
  where $\sigma \in \Sigma_n$, $f_i$, $g_i$ ($1 \leq i \leq n$) are
  morphisms of $\Delta S_+$ with specified sources and targets,
  $[m_i-1] \stackrel{f_i}{\to} [p_i-1] \stackrel{g_i}{\to} [q_i-1]$,
  and $V_i \in A^{\otimes m_i}$.  The map $\nu$ sends the
  element~(\ref{eq.D--algebra-invarianceL}) to $(\sigma_{\mathbf{q}}
  \bullet \mathbf{gf}^{\odot}) \otimes (V_{1} \otimes \cdots \otimes
  V_{n})$. On the other hand,
  element~(\ref{eq.D--algebra-invarianceL}) is equal (under $\Delta
  S_+$-equivariance) to:
  \begin{equation}\label{eq.D--algebra-invarianceR}
    \sigma \otimes ( g_1 \otimes \langle f_1, V_1 \rangle) \otimes
    \cdots \otimes ( g_n \otimes \langle f_n, V_n\rangle),
  \end{equation}
  and $\nu$ sends~(\ref{eq.D--algebra-invarianceR}) to:
  \begin{eqnarray}
    \label{eqn.mu(w)1}
    (\sigma_{\mathbf{q}} \bullet \mathbf{g}^{\odot} ) \otimes (
    \langle f_{1}, V_{1} \rangle \otimes \cdots \otimes \langle f_{n},
    V_{n} \rangle) &=& (\sigma_{\mathbf{q}} \bullet \mathbf{g}^{\odot}
    ) \otimes \langle \mathbf{f}^{\odot}, V_{1} \otimes \cdots \otimes
    V_{n} \rangle \\ 
    \label{eqn.mu(w)2}
    &=& (\sigma_{\mathbf{q}} \bullet
    \mathbf{gf}^{\odot}) \otimes (V_{1} \otimes \cdots \otimes V_{n}).
  \end{eqnarray}
\end{proof}

\begin{theorem}\label{thm.homology_operations}
  When the ground ring $k = \F_p$ for a prime $p$, symmetric homology
  $HS_*(A)$ admits Dyer-Lashof homology operations.
\end{theorem}
\begin{proof}
  This is an immediate result of
  Lemma~\ref{lem.E_infty-algebra_k-modules} and the fact that
  $HS_*(A)$ is the homology of $CA_*$.  The reader is referred to
  Dyer-Lashof~\cite{Dyer-Lashof}, May~\cite{M2}, or chapter I
  of~\cite{CLM}, for details on constructing the operations on any
  $E_{\infty}$ algebra.
\end{proof}

\section{Product Structure}
\label{sec.homology-operations}      

\subsection{Pontryagin Product}\label{sub.products}      

There is a well-defined graded-commutative product on the graded
$k$-module, $\{ HS_i(A)\}_{i \geq 0}$.
\begin{theorem}\label{thm.pontryagin}
  $HS_*(A)$ admits a Pontryagin product, giving it the structure of
  associative, graded commutative algebra.
\end{theorem}
\begin{proof}
  This follows directly from
  Lemma~\ref{lem.E_infty-algebra_k-modules}. The product is defined
  by:
  \begin{eqnarray*}
    (CA_*) \otimes (CA_*) &\hookrightarrow&
    \mathscr{D}_{\mathbf{Mod}}(2) \otimes_{\Sigma_2} (CA_*)^{\otimes
      2} \stackrel{\nu}{\to} (CA_*)\\ x \otimes y &\mapsto& c \otimes
    (x \otimes y) \mapsto \nu(c \otimes (x \otimes y)),
  \end{eqnarray*}
  where $\nu$ is defined in diagram~(\ref{eqn.passing_to_quotient}),
  and $c \in \mathscr{D}_{\mathbf{Mod}}(2)$ is a generator as a free
  $k$-module.
\end{proof}


\begin{corollary}\label{cor.trivialHS}
  Let $A$ be a unital associative $k$-algebra.  If the ideal
  generated by the commutator submodule is equal to the entire algebra
  ({\it i.e.} $\left( [A,A] \right) = A$), then $HS_*(A)$ is trivial
  in all degrees.
\end{corollary}
\begin{proof}
  $HS_0(A) = A/\left( [A,A] \right)$, so $HS_0(A)$ is trivial.  Now
  for any $x \in HS_q(A)$, we have $x = 1 \cdot x = 0 \cdot x$.
\end{proof}
\begin{remark} 
  It was pointed out in~\cite{A} that symmetric homology fails to
  preserve Morita equivalence.  Corollary~\ref{cor.trivialHS} shows
  the failure in a big way: $HS_*\left(M_n(A)\right)$ is trivial if
  $n > 1$.
\end{remark}

\begin{proposition}\label{prop.product-on-HS_0}
  When restricted to $HS_0(A) \otimes HS_0(A) \to HS_0(A)$, the
  Pontryagin product is the standard algebra multiplication map
  $A/\left( [A,A] \right) \otimes A/\left( [A,A] \right) \to A/\left(
  [A,A] \right)$.
\end{proposition}
\begin{proof}
  Examine the first few terms of the sequence, $0 \gets CA_0
  \stackrel{d_1}{\gets} CA_1$.  It is straightforward to verify that
  $d_1$ collapses the generators in degree $0$ to those of the form
  $([0] \gets [0]) \otimes a$ via the iterated multiplication map
  $A^{\otimes n} \to A$.
\end{proof}

\subsection{Explicit $HS_0(A)$-module structure of $HS_1(A)$}
\label{sub.module}      

The main result of this subsection is a concrete computation of the
Pontryagin product $HS_0(A) \otimes HS_1(A) \to HS_1(A)$.  We shall
need to induce the $\mathscr{D}_{\mathbf{Mod}}$-algebra structure of
$CA_*$ to the level of complexes in order to transfer the $E_{\infty}$
structure across a chain equivalence.  This step is trivial, as the
``chains'' functor of the Dold-Kan correspondence is lax monoidal.
However, we must remember to use the shuffle map when making
computations at the chain level.  Let $\overline{CA}_{\bullet}$ denote
the Moore complex of $CA_*$.
\begin{lemma}\label{lem.E_infty-algebra_chains}
  The $\mathscr{D}_{\mathbf{Mod}}$-algebra structure on $CA_*$ induces
  a $\mathscr{D}_{\mathbf{Ch}_{\bullet}^{+}}$-algebra structure on
  $\overline{CA}_{\bullet}$.
\end{lemma}
We shall also need some machinery from~\cite{A}, \S\S~10-11.  For each
$n \geq -1$, define the projective $\Delta S_+^{\mathrm{op}}$-module
$(\Delta S_+)_n$ as in~(\ref{eqn.P_n-def}).  The following sequence is
a partial resolution of $\underline{k}$ by projective $\Delta
S_+^{\mathrm{op}}$-modules:
\begin{equation}\label{eqn.partial_resolution_P}
  \begin{tikzcd}
    \underline{k} & \lar[swap]{\epsilon} (\Delta S_+)_0 &
    \lar[swap]{\delta} (\Delta S_+)_2 & \lar[swap]{(\alpha, \beta)}
    (\Delta S_+)_3 \oplus (\Delta S_+)_0,
  \end{tikzcd}
\end{equation}
in which $\epsilon(f) = 1$ for any morphism $f \colon [n] \to [0]$,
$\delta(f) = (x_0x_1x_2)\circ f - (x_2x_1x_0)\circ f$, $\alpha(f) =
(x_0x_1 \otimes x_2 \otimes x_3)\circ f + (x_3 \otimes x_2x_0 \otimes
x_1)\circ f + (x_1x_2x_0 \otimes 1 \otimes x_3)\circ f + (x_3 \otimes
x_1x_2 \otimes x_0) \circ f$, and $\beta(f) = (1 \otimes x_0 \otimes
1)\circ f$.  Thus, there is a small partial chain complex that
computes $HS_i(A)$ for $i = 0, 1$:
\begin{equation}\label{eqn.partial_complex}
  \begin{tikzcd}
    0 & \lar A & \lar[swap]{\partial_1} A^{\otimes 3} &
    \lar[swap]{\partial_2} A^{\otimes 4} \oplus A,
  \end{tikzcd}
\end{equation}
in which
\begin{eqnarray*}
  \partial_1(a \otimes b \otimes c) &=& abc - cba \\ \partial_2(a
  \otimes b \otimes c \otimes d, e) &=& ab \otimes c \otimes d + d
  \otimes ca \otimes b + bca \otimes 1 \otimes d + d \otimes bc
  \otimes a + 1 \otimes e \otimes 1.
\end{eqnarray*}
If $a \in A$, denote by $[a]$ the corresponding element of $HS_0(A)$,
and if $a \otimes b \otimes c \in A^{\otimes 3}$, denote by $[a
  \otimes b \otimes c]$ the corresponding element of $HS_1(A)$.

Our first goal is to set up an explicit equivalence between the
partial complex~(\ref{eqn.partial_complex}) and
$\overline{CA}_{\bullet}$, at least up to degree 1, and then use the
equivalence to give a concrete formula for the product structure.  In
Diagram~(\ref{eqn.diagram_equiv}), the differential $d$ is induced
from the simplicial face maps.  Below, we define and discuss the maps
$F_i$ and $G_i$ for $i=0,1$.
\begin{equation}\label{eqn.diagram_equiv}
  \begin{tikzcd}
    0 & A \lar \dar[bend left]{F_0} & A^{\otimes 3}
    \lar[swap]{\partial_1} \dar[bend left]{F_1}& A^{\otimes 4} \oplus
    A \lar[swap]{\partial_2} \\ 0 & \overline{CA}_0 \lar \uar[bend
      left]{G_0} & \overline{CA}_1 \lar[swap]{d_1} \uar[bend
      left]{G_1} & \overline{CA}_2 \lar[swap]{d_2}
  \end{tikzcd}
\end{equation}
For each $m \geq -1$, let $\pi_m \colon [m] \to [0]$ be the unique
order-perserving $\Delta S_+$ morphism, and $\rho_m \colon [m] \to
[0]$ be the unique order-reversing $\Delta S_+$ morphism.  For
convenience, let $\mathbf{a} = a_0 \otimes \cdots \otimes a_n$ stand
for an arbitrary element of $A^{\otimes(n+1)}$.  We define the maps
$F_0$ and $G_0$ as follows:
\begin{eqnarray}
  F_0(a) &\stackrel{def}{=}& \left([0] \stackrel{\mathrm{id}}{\gets}
  [0]\right) \otimes a\\ G_0\left( \left([m] \stackrel{f}{\gets}
  [n]\right) \otimes \mathbf{a}\right) &\stackrel{def}{=}& \langle
  \pi_{m}f, \mathbf{a} \rangle
\end{eqnarray}
Observe that $G_0F_0(a) = a$.  To show $F_0G_0 \simeq \mathrm{id}$,
define a homotopy map:
\begin{eqnarray}
  h_0 \colon \overline{CA}_0 &\to& \overline{CA}_1 \\ \left([m]
  \stackrel{f}{\gets} [n]\right) \otimes \mathbf{a} &\mapsto&
  \left([0] \stackrel{\pi_m}{\gets} [m] \stackrel{f}{\gets} [n]\right)
  \otimes \mathbf{a}
\end{eqnarray}
Observe,
\begin{eqnarray}
  d_1h_0\left( \left([m] \stackrel{f}{\gets} [n]\right) \otimes
  \mathbf{a} \right) &=& \left( \left([0] \stackrel{\pi_mf}{\gets} [n]
  \right) \otimes \mathbf{a} \right) - \left( \left([m]
  \stackrel{f}{\gets} [n] \right) \otimes \mathbf{a} \right) \\ &=&
  \left( \left([0] \stackrel{\mathrm{id}}{\gets} [0] \right) \otimes
  \langle \pi_mf, \mathbf{a} \rangle \right) - \left( \left([m]
  \stackrel{f}{\gets} [n] \right) \otimes \mathbf{a} \right) \\&=&
  (F_0G_0 - \mathrm{id})\left( \left([m] \stackrel{f}{\gets}
           [n]\right) \otimes \mathbf{a} \right).
\end{eqnarray}
Next, define $F_1$:
\begin{equation}
  F_1(a \otimes b \otimes c) \stackrel{def}{=} \left[\left([0]
    \stackrel{\pi_2}{\gets} [2] \stackrel{\mathrm{id}}{\gets} [2]
    \right) -\left([0] \stackrel{\rho_2}{\gets} [2]
    \stackrel{\mathrm{id}}{\gets} [2] \right)\right] \otimes (a
  \otimes b \otimes c)
\end{equation}
The maps $F_0, F_1$ are compatible with the differentials, as
illustrated by a diagram-chase.
\begin{equation}
  \begin{tikzcd}[column sep=large]
    abc - bca \arrow[mapsto]{dd}{F_0} &
    \arrow[mapsto]{l}[swap]{\partial_1} a \otimes b \otimes c
    \arrow[mapsto]{d}{F_1}\\ & \left[\left([0] \stackrel{\pi_2}{\gets}
      [2] \stackrel{\mathrm{id}}{\gets} [2] \right) -\left([0]
      \stackrel{\rho_2}{\gets} [2] \stackrel{\mathrm{id}}{\gets} [2]
      \right)\right] \otimes (a \otimes b \otimes c)
    \arrow[mapsto]{d}[swap]{d_1} \\ \left( [0]
    \stackrel{\mathrm{id}}{\gets} [0] \right) \otimes (abc - bca)
    \rar[equals]& \left[\left( [0] \stackrel{\pi_2}{\gets} [0] \right)
      - \left( [0] \stackrel{\rho_2}{\gets} [0] \right)\right] \otimes
    (a \otimes b \otimes c)
  \end{tikzcd}
\end{equation}
Defining $G_1$ is a bit trickier.  For each $n \geq 0$, construct a
quiver $\widetilde{\mathscr{G}}_n$ as follows: The vertices of
$\widetilde{\mathscr{G}}_n$ are permutations of $\{0, 1, \ldots, n\}$.
The edges of $\widetilde{\mathscr{G}}_n$ are in one-to-one
correspondence with the elements of $\Delta S_+([n], [2])$.  For any
$f \colon [n] \to [m]$, write $f = \left(\phi(f), \gamma(f)\right)$
for the unique $\Delta S_+$ factorization.  Now each $f$ labels an
edge in $\widetilde{\mathscr{G}}_n$ whose source is the permuation
$\gamma(f) = \gamma(\pi_2 f)$ and whose target is $\gamma(\rho_2 f)$.
For example, in $\widetilde{\mathscr{G}}_5$, the morphism $x_3x_1
\otimes x_4 \otimes x_0x_5$ (written in tensor notation) labels an
edge from vertex ``$31405$'' to vertex ``$05431$.''  Let
$\mathscr{G}_{n}$ be a maximal subtree of $\widetilde{\mathscr{G}}_n$.
Note, $\mathscr{G}_n$ is connected, which is a result of the fact that
$k \gets (\Delta S_+)_0([n]) \gets (\Delta S_+)_2([n])$ is exact for
all $n \geq 0$ (see~(\ref{eqn.partial_resolution_P}) and~\cite{A}
Lemma~79).  The purpose of $\mathscr{G}_n$ is to record ways in which
one permutation may be converted to any other by way of block
permutations of no more than three blocks at a time.
\begin{example}
  $\mathscr{G}_2$ may be chosen to be the graph on vertices $01$ and
  $10$ with a single edge $01 \to 10$ labeled by $x_0 \otimes x_1
  \otimes 1$.  See Figures~\ref{diag.G_2} and~\ref{diag.G_3} for
  further examples (for brevity in the diagrams, we may write
  morphisms of $\Delta S_+$ in tensor notation using the symbols $a =
  x_0$, $b=x_1$, $c=x_2$, $d = x_3$, etc.).
\end{example}

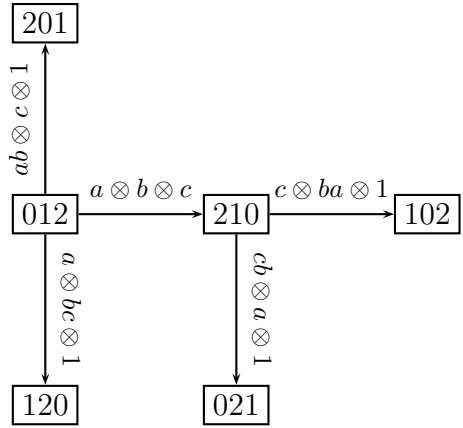
\begin{figure}[!htp]
  \begin{center}
    \psset{unit=1in}
    \begin{pspicture}(3,2)
      \rput(0,1){\rnode{012}{\psframebox{$012$}}}
      \rput(0,2){\rnode{201}{\psframebox{$201$}}}
      \rput(0,0){\rnode{120}{\psframebox{$120$}}}
      \rput(1,1){\rnode{210}{\psframebox{$210$}}}
      \rput(2,1){\rnode{102}{\psframebox{$102$}}}
      \rput(1,0){\rnode{021}{\psframebox{$021$}}}
      \ncline{->}{012}{201} \aput{:U}{{\footnotesize $ab \otimes c
          \otimes 1$}} \ncline{->}{012}{210} \aput{:U}{{\footnotesize
          $a \otimes b \otimes c$}} \ncline{->}{012}{120}
      \aput{:U}{{\footnotesize $a \otimes bc \otimes 1$}}
      \ncline{->}{210}{021} \aput{:U}{{\footnotesize $cb \otimes a
          \otimes 1$}} \ncline{->}{210}{102} \aput{:U}{{\footnotesize
          $c \otimes ba \otimes 1$}}
    \end{pspicture}\caption[$\mathscr{G}_{2}$]{One possible choice 
      of $\mathscr{G}_{2}$} \label{diag.G_2}
  \end{center}
\end{figure}

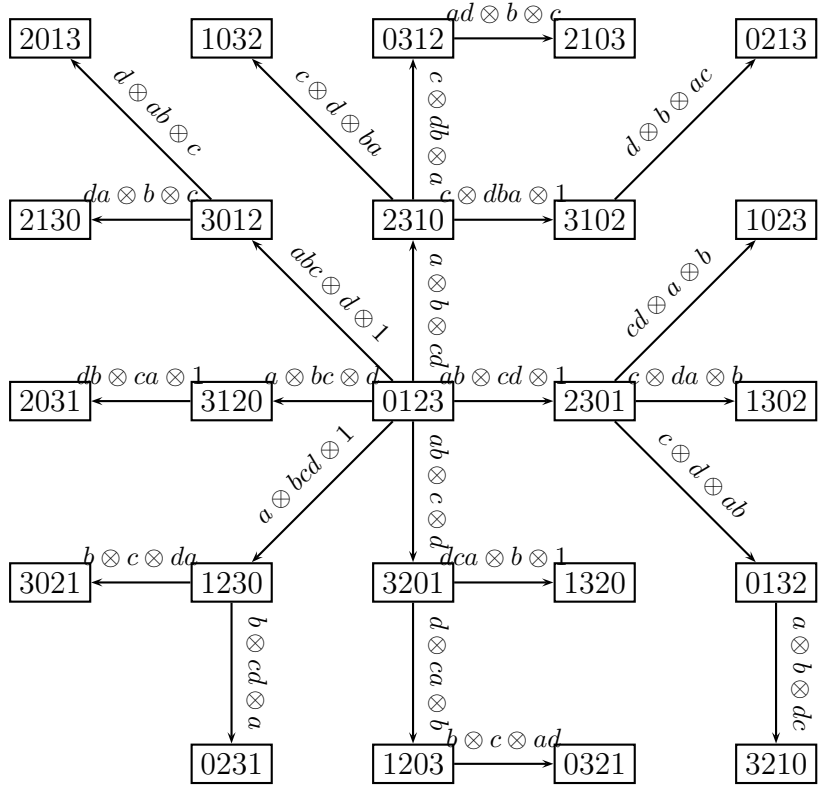
\begin{figure}[!htp]
  \begin{center}
    \psset{unit=0.95in}
    \begin{pspicture}(4,4.1)
      \rput(0,4){\rnode{cabd}{\psframebox{$2013$}}}
      \rput(1,4){\rnode{badc}{\psframebox{$1032$}}}
      \rput(2,4){\rnode{adbc}{\psframebox{$0312$}}}
      \rput(3,4){\rnode{cbad}{\psframebox{$2103$}}}
      \rput(4,4){\rnode{acbd}{\psframebox{$0213$}}}
      \rput(0,3){\rnode{cbda}{\psframebox{$2130$}}}
      \rput(1,3){\rnode{dabc}{\psframebox{$3012$}}}
      \rput(2,3){\rnode{cdba}{\psframebox{$2310$}}}
      \rput(3,3){\rnode{dbac}{\psframebox{$3102$}}}
      \rput(4,3){\rnode{bacd}{\psframebox{$1023$}}}
      \rput(0,2){\rnode{cadb}{\psframebox{$2031$}}}
      \rput(1,2){\rnode{dbca}{\psframebox{$3120$}}}
      \rput(2,2){\rnode{abcd}{\psframebox{$0123$}}}
      \rput(3,2){\rnode{cdab}{\psframebox{$2301$}}}
      \rput(4,2){\rnode{bdac}{\psframebox{$1302$}}}
      \rput(0,1){\rnode{dacb}{\psframebox{$3021$}}}
      \rput(1,1){\rnode{bcda}{\psframebox{$1230$}}}
      \rput(2,1){\rnode{dcab}{\psframebox{$3201$}}}
      \rput(3,1){\rnode{bdca}{\psframebox{$1320$}}}
      \rput(4,1){\rnode{abdc}{\psframebox{$0132$}}}
      \rput(1,0){\rnode{acdb}{\psframebox{$0231$}}}
      \rput(2,0){\rnode{bcad}{\psframebox{$1203$}}}
      \rput(3,0){\rnode{adcb}{\psframebox{$0321$}}}
      \rput(4,0){\rnode{dcba}{\psframebox{$3210$}}}
      \ncline{<-}{bcda}{abcd} \aput{:U}{{\footnotesize $a \otimes bcd
          \otimes 1$}} \ncline{<-}{dbca}{abcd}
      \aput{:U}{{\footnotesize $a \otimes bc \otimes d$}}
      \ncline{->}{abcd}{dcab} \aput{:U}{{\footnotesize $ab \otimes c
          \otimes d$}} \ncline{<-}{cdba}{abcd}
      \aput{:U}{{\footnotesize $a \otimes b \otimes cd$}}
      \ncline{<-}{dabc}{abcd} \aput{:U}{{\footnotesize $abc \otimes d
          \otimes 1$}} \ncline{->}{dcab}{bdca}
      \aput{:U}{{\footnotesize $dca \otimes b \otimes 1$}}
      \ncline{->}{adbc}{cbad} \aput{:U}{{\footnotesize $ad \otimes b
          \otimes c$}} \ncline{<-}{cadb}{dbca}
      \aput{:U}{{\footnotesize $db \otimes ca \otimes 1$}}
      \ncline{->}{cdba}{dbac} \aput{:U}{{\footnotesize $c \otimes dba
          \otimes 1$}} \ncline{<-}{badc}{cdba}
      \aput{:U}{{\footnotesize $c \otimes d \otimes ba$}}
      \ncline{<-}{adbc}{cdba} \aput{:U}{{\footnotesize $c \otimes db
          \otimes a$}} \ncline{->}{dcab}{bcad}
      \aput{:U}{{\footnotesize $d \otimes ca \otimes b$}}
      \ncline{<-}{cabd}{dabc} \aput{:U}{{\footnotesize $d \otimes ab
          \otimes c$}} \ncline{<-}{cbda}{dabc}
      \aput{:U}{{\footnotesize $da \otimes b \otimes c$}}
      \ncline{->}{abcd}{cdab} \aput{:U}{{\footnotesize $ab \otimes cd
          \otimes 1$}} \ncline{->}{bcda}{acdb}
      \aput{:U}{{\footnotesize $b \otimes cd \otimes a$}}
      \ncline{<-}{dacb}{bcda} \aput{:U}{{\footnotesize $b \otimes c
          \otimes da$}} \ncline{->}{cdab}{bacd}
      \aput{:U}{{\footnotesize $cd \otimes a \otimes b$}}
      \ncline{->}{cdab}{bdac} \aput{:U}{{\footnotesize $c \otimes da
          \otimes b$}} \ncline{->}{cdab}{abdc}
      \aput{:U}{{\footnotesize $c \otimes d \otimes ab$}}
      \ncline{->}{abdc}{dcba} \aput{:U}{{\footnotesize $a \otimes b
          \otimes dc$}} \ncline{->}{dbac}{acbd}
      \aput{:U}{{\footnotesize $d \otimes b \otimes ac$}}
      \ncline{->}{bcad}{adcb} \aput{:U}{{\footnotesize $b \otimes c
          \otimes ad$}}
    \end{pspicture}\caption[$\mathscr{G}_{3}$]{One possible choice 
      of $\mathscr{G}_3$} \label{diag.G_3}
  \end{center}
\end{figure}

Consider a typical element $( [p] \stackrel{g}{\gets} [m]
\stackrel{f}{\gets} [n]) \otimes \mathbf{a} \in \overline{CA}_2$.
There is a unique path from $\gamma(gf)$ to $\gamma(f)$ in
$\mathscr{G}_n$.  Let $\mathrm{Path}(gf, f)$ be the set of edge
labels, each taken to be positive or negative depending on the
direction of the arrow as one proceeds from $\gamma(gf)$ to
$\gamma(f)$ in the tree (positive if with the arrow; negative if
against it).  If $\gamma(gf) = \gamma(f)$, then $\mathrm{Path}(gf, f)
= \emptyset$.  Define $G_1 \colon \overline{CA}_1 \to A^{\otimes 3}$
thus:
\begin{equation}
  G_1\left(( [p] \stackrel{g}{\gets} [m] \stackrel{f}{\gets} [n])
  \otimes \mathbf{a} \right) = \sum_{e \in \mathrm{Path}(gf, f)}
  \langle e, \mathbf{a} \rangle.
\end{equation}
Note, the choice of maximal subtree $\mathscr{G}_n$ for each $n$ must
be made once and not changed, as different choices for subtree will
affect the definition of $G_1$.
\begin{example}
  Let $f = 1 \otimes x_2 \otimes x_0 \otimes 1 \otimes x_1 \colon [2]
  \to [4]$ and $g = x_3 \otimes x_2x_0 \otimes 1 \otimes x_1x_4 \colon
      [4] \to [3]$.  Then $gf = 1 \otimes x_0 \otimes 1 \otimes
      x_2x_1$, and $\gamma(gf) = 021$ is the ``start'' node, while
      $\gamma(f) = 201$ is the ``end'' node. We use
      Figure~\ref{diag.G_2} to determine the path.
  \begin{equation}
    G_1\left(( [3] \stackrel{g}{\gets} [4] \stackrel{f}{\gets} [2])
    \otimes (a \otimes b \otimes c) \right) = -cb \otimes a \otimes 1
    - a \otimes b \otimes c + ab \otimes c \otimes 1.
  \end{equation}
\end{example}
The maps $G_0, G_1$ are also compatible with the differentials,
as we verify below:
\begin{equation}
  G_0d_1\left( ( [p] \stackrel{g}{\gets} [m] \stackrel{f}{\gets} [n])
  \otimes \mathbf{a}\right) = \langle \pi_pgf, \mathbf{a} \rangle -
  \langle \pi_m f, \mathbf{a} \rangle.
\end{equation}
\begin{eqnarray}
  \partial_1G_1 \left( ( [p] \stackrel{g}{\gets} [m]
  \stackrel{f}{\gets} [n]) \otimes \mathbf{a}\right) &=& \partial_1
  \left(\sum_{e} \langle e, \mathbf{a}
  \rangle \right)\\ &=& \sum_{e} \left( \langle \pi_2 e, \mathbf{a}
  \rangle - \langle \rho_2 e, \mathbf{a} \rangle \right)
\end{eqnarray}
The sum telescopes so that only the start and end vertices of the path
remain: $\langle \pi_pgf, \mathbf{a} \rangle - \langle \pi_m f,
\mathbf{a} \rangle$.

Using $\mathscr{G}_2$ as in Figure~\ref{diag.G_2}, we find that
$G_1F_1 = \mathrm{id}$.  The verification is provided below (here,
$\mathbf{a} = a\otimes b \otimes c$, and observe that $\gamma(\rho_2)
=$ ``$210$'' in Figure~\ref{diag.G_2}, so $\mathrm{Path}(\rho_2,
\mathrm{id}) = \{ -(a\otimes b \otimes c) \}$).
\begin{eqnarray}
  G_1F_1(\mathbf{a}) &=& G_1\left( \left( [0] \stackrel{\pi_2}{\gets}
  [2] \stackrel{\mathrm{id}}{\gets} [2]\right) \otimes
  \mathbf{a}\right) - G_1\left( \left( [0] \stackrel{\rho_2}{\gets}
         [2] \stackrel{\mathrm{id}}{\gets} [2]\right) \otimes
         \mathbf{a}\right)\\ &=& \sum_{e \in \mathrm{Path}(\pi_2,
           \mathrm{id})} \langle e, \mathbf{a} \rangle - \sum_{e \in
           \mathrm{Path}(\rho_2, \mathrm{id})} \langle e, \mathbf{a}
         \rangle \\ &=& 0 - \left( - (a \otimes b \otimes c) \right)
         \\ &=& \mathbf{a}.
\end{eqnarray}
Finally, we set up a homotopy map $h_1 \colon \overline{CA}_1 \to
\overline{CA}_2$ to show that $F_1G_1 \simeq \mathrm{id}$.
\begin{equation}\label{eqn.homotopy}
  \begin{tikzcd}[column sep=large]
      \overline{CA}_0 \rar[bend left]{h_0} & \overline{CA}_1
      \arrow[loop above]{}[name=F1G1]{F_1G_1} \lar[swap]{d_1}
      \rar[bend left]{h_1} & \overline{CA}_2 \lar[swap]{d_2}
  \end{tikzcd}
\end{equation}

At this point, it is helpful to use an abbreviated notation for
$i$-chains of $\overline{CA}_{\bullet}$:
\begin{equation}\label{eqn.abbrev}
  ( g_i, \ldots, g_1, f, \mathbf{a}) \stackrel{def}{=} ( [m_i]
  \stackrel{g_i}{\gets} \cdots \stackrel{g_1}{\gets} [m_1]
  \stackrel{f}{\gets} [n]) \otimes \mathbf{a}.
\end{equation}
Of course, $\Delta S_+$-equivariance still applies; in particular, the
element~(\ref{eqn.abbrev}) is equal to $(g_i, \ldots, g_1,
\mathrm{id}, \langle f, \mathbf{a} \rangle)$.  Let $(g, f, \mathbf{a})
\in \overline{CA}_1$ be as above, that is, $[p] \stackrel{g}{\gets}
    [m] \stackrel{f}{\gets} [n]$ is a sequence of $\Delta S_+$
    morphisms.
\begin{eqnarray}
  \label{eqn.F_1G_1(1)}
  F_1G_1(g, f, \mathbf{a}) &=& \sum_{e \in \mathrm{Path}(gf, f)}
  \left[ \left(\pi_2, \mathrm{id}, \langle e, \mathbf{a} \rangle\right)
    - \left(\rho_2, \mathrm{id}, \langle e, \mathbf{a} \rangle\right)
    \right] \\ 
    \label{eqn.F_1G_1(2)}
  &=& \sum_{e \in \mathrm{Path}(gf, f)} \left[ \left(\pi_2,
    e, \mathbf{a} \right) - \left(\rho_2, e, \mathbf{a}\right) \right].
\end{eqnarray}
\begin{equation}\label{eqn.h_0d_1}
  h_0d_1(g, f, \mathbf{a}) = (\pi_p, gf, \mathbf{a}) - (\pi_m, f,
  \mathbf{a}).
\end{equation}
We define $h_1$ by:
\begin{equation}\label{eqn.h_1}
  h_1 \colon (g, f, \mathbf{a}) \mapsto (\pi_m, f, \mathrm{id},
  \mathbf{a}) - (\pi_pg, f, \mathrm{id}, \mathbf{a}) - (\pi_p, g, f,
  \mathbf{a}) + \sum_{e \in \mathrm{Path}(gf, f)} \left[ \left(\pi_2,
    e, \mathrm{id}, \mathbf{a} \right) - \left(\rho_2, e, \mathrm{id},
    \mathbf{a} \right)\right].
\end{equation}
Then a tedious but straighforward calculation shows that $F_1G_1 -
\mathrm{id} = d_2h_1 + h_0d_1$.  Some details are shown below, as
$d_2$ is applied to the various terms that comprise the right hand
side of~(\ref{eqn.h_1}).
\begin{eqnarray}
  \label{eqn.d_2h_1(1)}
  d_2 \colon (\pi_m, f, \mathrm{id}, \mathbf{a}) &\mapsto& (\pi_m, f,
  \mathbf{a}) - (\pi_mf, \mathrm{id}, \mathbf{a}) + (f, \mathrm{id},
  \mathbf{a}). \\
  \label{eqn.d_2h_1(2)}
  d_2 \colon (\pi_pg, f, \mathrm{id}, \mathbf{a}) &\mapsto& (\pi_pg,
  f, \mathbf{a}) - (\pi_pgf, \mathrm{id}, \mathbf{a}) + (f,
  \mathrm{id}, \mathbf{a}). \\
  \label{eqn.d_2h_1(3)}
  d_2 \colon (\pi_p, g, f, \mathbf{a}) &\mapsto& (\pi_p, gf,
  \mathbf{a}) - (\pi_pg, f, \mathbf{a}) + (g, f, \mathbf{a}). \\
  \label{eqn.d_2h_1(4)}
  d_2 \colon (\pi_2, e, \mathrm{id}, \mathbf{a}) &\mapsto& (\pi_2, e,
  \mathbf{a}) - (\pi_2e, \mathrm{id}, \mathbf{a}) + (e, \mathrm{id},
  \mathbf{a}).\\
  \label{eqn.d_2h_1(5)}
  d_2 \colon (\rho_2, e, \mathrm{id}, \mathbf{a}) &\mapsto& (\rho_2,
  e, \mathbf{a}) - (\rho_2e, \mathrm{id}, \mathbf{a}) + (e,
  \mathrm{id}, \mathbf{a}).
\end{eqnarray}
In view of~(\ref{eqn.h_1}) and
(\ref{eqn.d_2h_1(1)})--(\ref{eqn.d_2h_1(5)}), and after many
cancellations,
\begin{equation}
  d_2h_1(g, f, \mathbf{a}) = (\pi_m, f, \mathbf{a}) - (\pi_p, gf,
  \mathbf{a}) - (g, f, \mathbf{a}) + \sum_{e} \left[ (\pi_2, e,
    \mathbf{a}) - (\rho_2, e, \mathbf{a})\right],
\end{equation}
which is the same as (cf.~Eqns.~(\ref{eqn.F_1G_1(2)})
and~(\ref{eqn.h_0d_1})):
\begin{equation}
   (-h_0d_1 - \mathrm{id} + F_1G_1)(g, f, \mathbf{a}).
\end{equation}

\begin{proposition}\label{prop.module-structure}
  For a unital associative algebra $A$ over commutative ground ring
  $k$, $HS_1(A)$ is a left $HS_0(A)$--module, via
  \begin{equation}
    [a] \bullet [b \otimes c \otimes d] = [ab \otimes c \otimes d] -
    [b \otimes ca \otimes d] + [b \otimes c \otimes ad].
  \end{equation}
  Moreover, there is a right module structure,
  \begin{equation}
    [b\otimes c \otimes d] \bullet [a] = [ba \otimes c \otimes d] - [b
      \otimes ac \otimes d] + [b \otimes c \otimes da],
  \end{equation}
  and the two actions agree in the sense that $[a] \bullet [b\otimes c
    \otimes d] = [b\otimes c \otimes d] \bullet [a]$.
\end{proposition}
\begin{remark}
  This module structure was first discovered on the chain level before
  the Pontryagin product was discovered.  Below is the explicit
  derivation using Theorem~\ref{thm.pontryagin}.
\end{remark}
\begin{proof}
  Let $w, x, y, z \in A$, so that $w$ represents a $0$-chain and $x
  \otimes y \otimes z$ represents a $1$-chain in the partial
  sequence~(\ref{eqn.partial_complex}) used to compute $HS_*(A)$.
  Consider $\mathrm{id}_{\Sigma_2} \in
  \mathscr{D}_{\mathbf{Ch}_{\bullet}^{+}}(2)$, and let $F_*$ and $G_*$
  be the chain equivalences developed above.  Note, in
  line~(\ref{eqn.product4}), morphisms of $\Delta S_+$ are written in
  tensor notation.
  \begin{eqnarray}
    \label{eqn.product1}
    \lefteqn{\mathrm{id_{\Sigma_2}} \otimes (a) \otimes (b \otimes c
      \otimes d)} \\
    \label{eqn.product2}
    &\stackrel{\mathrm{id} \otimes F_*^{\otimes 2}}{\mapsto}&
    \mathrm{id_{\Sigma_2}} \otimes (\mathrm{id}_{[0]}, a) \otimes
    \left( (\pi_2, \mathrm{id}_{[2]}, b \otimes c \otimes d) -
    (\rho_2, \mathrm{id}_{[2]}, b \otimes c \otimes d)\right) \\
    \label{eqn.product3}
    &\stackrel{\nu}{\mapsto}& \left[(\mathrm{id}_{[0]} \odot \pi_2,
      \mathrm{id}_{[3]}) - (\mathrm{id}_{[0]} \odot \rho_2,
      \mathrm{id}_{[3]})\right] \otimes (a \otimes b \otimes c \otimes
    d)\\
    \label{eqn.product4}
    &=& \left[(x_0 \otimes x_1x_2x_3, \mathrm{id}_{[3]}) - (x_0
      \otimes x_3x_2x_1, \mathrm{id}_{[3]}) \right] \otimes (a \otimes
    b \otimes c \otimes d)\\
    \label{eqn.product5}
    &\stackrel{G_1}{\mapsto}& b \otimes c \otimes ad + d \otimes ca
    \otimes b + ab \otimes c \otimes d
  \end{eqnarray}
  Finally, using the {\it sign relation} (see~\cite{A}, \S10), we
  have equality in $HS_1(A)$:
  \begin{equation}
    [b \otimes c \otimes ad] + [d \otimes ca \otimes b] + [ab \otimes
      c \otimes d] = [ab \otimes c \otimes d] - [b \otimes ca \otimes
      d] + [b \otimes c \otimes ad].
  \end{equation}
  The product $HS_1(A) \otimes HS_0(A) \to HS_1(A)$ can be found
  explicitly in a similar manner.  The fact that the two products
  agree follows from the observation that their difference is a
  boundary.
\end{proof}
\begin{remark}
  Theoretically, if the resolution~(\ref{eqn.partial_resolution_P})
  could be extended further, then one could extend the maps $F_i$ and
  $G_i$ to higher degrees in order to study the product structure of
  $HS_*(A)$.  However, this tedious ``nuts-and-bolts'' approach does
  not seem to offer best ratio of payoff in exchange for the work put
  in.
\end{remark}

\subsection{Computed Results}
Using \verb|GAP|, the following explicit computations of the
$HS_0(A)$--module structure on $HS_1(A)$ were made for some
$\Z$-algebras.  Note in each case below, $HS_0(A) = A$ since $A$ is
commutative.

\begin{center}
\begin{tabular}{l|l|l}
  $A$ & $HS_1(A \;|\; \Z)$ & $HS_0(A)$--module structure\\ 
  \hline
  $\Z[t]/(t^2)$ & $\Z/2\Z \oplus \Z/2\Z$ & Generated by $u$ with
  $2u=0$\\ 
  $\Z[t]/(t^3)$ & $\Z/2\Z \oplus \Z/2\Z$ & Generated by $u$
  with $2u=0$ and $t^2u=0$\\ 
  $\Z[t]/(t^4)$ & $(\Z/2\Z)^4$ & Generated
  by $u$ with $2u=0$\\ 
  \hline 
  $\Z[C_2]$ & $\Z/2\Z \oplus \Z/2\Z$ &
  Generated by $u$ with $2u=0$\\ 
  $\Z[C_3]$ & $0$ & \\ 
  $\Z[C_4]$ & $(\Z/2\Z)^4$ & Generated by $u$ with $2u=0$\\ 
  $\Z[C_5]$ & $0$ &
  \\ \hline
\end{tabular}
\end{center}


\end{document}